\documentclass{amsart}
\pdfoutput=1

\usepackage[T1]{fontenc}
\usepackage[utf8]{inputenc}
\usepackage[english]{babel}
\usepackage{lmodern}
\usepackage{a4wide}
\usepackage{graphicx}
\usepackage{color}
\usepackage{changepage}
\usepackage{tikz}

\usepackage{amsfonts,amsmath,amssymb,amscd}
\usepackage{mathrsfs}

\usepackage[numbers]{natbib}
\usepackage[nameinlink]{cleveref}
\usepackage{thmtools}

\numberwithin{equation}{section}

\declaretheoremstyle[
  headfont=\bfseries,
  bodyfont=\itshape,
  spaceabove=9pt,
  spacebelow=9pt,
]{thmstyle}

\declaretheoremstyle[
  headfont=\bfseries,
  bodyfont=\normalfont,
  spaceabove=9pt,
  spacebelow=9pt,
]{defstyle}

\declaretheorem[
  style=thmstyle,
  numberwithin=section,
  name=Theorem
]{thm}

\declaretheorem[style=thmstyle, sibling=thm, name=Lemma]{lemma}
\declaretheorem[style=thmstyle, sibling=thm, name=Proposition]{prop}
\declaretheorem[style=thmstyle, sibling=thm, name=Corollary]{corollary}
\declaretheorem[style=thmstyle, sibling=thm, name=Observation]{observation}
\declaretheorem[style=thmstyle, sibling=thm, name=Notation]{notation}

\declaretheorem[style=thmstyle, numbered=no, name=Theorem]{theorem*}
\declaretheorem[style=thmstyle, numbered=no, name=Lemma]{lemma*}

\declaretheorem[style=defstyle, sibling=thm, name=Definition]{dfn}

\declaretheorem[style=defstyle, numbered=no, name=Definition]{dfn*}
\declaretheorem[style=defstyle, numbered=no, name=Remark]{rem*}
\declaretheorem[style=defstyle, numbered=no, name=Note]{note*}

\declaretheorem[style=thmstyle, name=Theorem]{innercustomthm}
\newenvironment{customthm}[1]
  {\renewcommand\theinnercustomthm{#1}\innercustomthm}
  {\endinnercustomthm}

\newcounter{Cequ}
\newenvironment{CEquation}
  {\stepcounter{Cequ}%
   \addtocounter{equation}{-1}%
   \equation}
  {\endequation}

\newcommand{\N}{\mathbb{N}}

\title{On the conditions for the continuity of the Hausdorff measure}
\author{Rafał Tryniecki}
\address{Rafał Tryniecki:University of Warsaw,Institute of Mathematics, ul. Banacha 2, 02-097 Warszawa, Poland}
\email{rafal.tryniecki@mimuw.edu.pl}
\thanks {This research was funded in whole or in part by the National Science Centre, Poland, grant no. 2023/49/B/ST1/03015.} %For the purpose of Open Access, the authors have applied a CC-BY public copyright licence to any Author Accepted Manuscript (AAM) version arising from this submission.}

\begin{document}

\begin{abstract}
Let $(b_k)_{k = 0}^\infty$ be strictly decreasing sequence of real numbers such that $b_0 = 1$ and $\{f_k:[b_k,b_{k-1}]\to [0,1]\}_{k\in\N}$ be decreasing functions such that $f_k(b_k) = 1$ and $f_k(b_{k-1}) = 0$, $k = 1, 2, \dots$. By $g_k: [0,1] \to [b_k, b_{k-1}]$ we denote the inverse of $f_k$ for  $k = 1,2 \dots$.
First, we define iterated function system (IFS) $S_n$ by limiting the collection of functions $g_k$ to first n, meaning $S_n = \{g_k \}_{\{k=1, \dots n\}}$. Let $J_n$ denote the limit set of $S_n$.
In the first part, we show that if $S_n$ fulfills the following two conditions:  (1)~$\lim\limits_{n \to \infty} \left(1-h_n\right) \ln{n} = 0 $ where $h_n$ is the Hausdorff dimension of $J_n$, and (2)~$\sup \limits_{k\in \mathbb{N}} \left \{\frac{b_k-b_{k+1}}{b_{k+1}} \right \} < \infty $, then $\lim\limits_{n\to \infty} H_{h_n}(J_n) = 1 = H_1(J)$, where $h_n$ is the Hausdorff dimension of $J_n$ and $H_{h_n}$ is the corresponding Hausdorff measure. 
In the second part, we provide four conditions for IFS consisting of nonlinear functions $f_k$ which guarantee that  $\lim\limits_{n\to \infty} H_{h_n}(J_n) = 1 = H_1(J)$, where $h_n$ is the Hausdorff dimension of $J_n$ and $H_{h_n}$ is the corresponding Hausdorff measure. We also provide a wide collection of examples of families of IFSes fulfilling those assumptions.
\end{abstract}
\maketitle
\section{Introduction}
\par 
Let $g_k: [0,1] \to [\frac{1}{k+1}, \frac{1}{k}]$ be a collection of maps given by $g_k(x) = \frac{1}{k+x}$ \mbox{($k = 1, 2, \dots$)}. The collection of inverse maps $f_k:[\frac{1}{k+1},\frac{1}{k}] \to [0,1]$ is given by $f_k(x) = \{\frac{1}{x}\}$. Restricting each map $f_k$ to the half-open interval $(\frac{1}{k+1}, \frac{1}{k}]$ and putting $f = f_k(x)$ for $x \in (\frac{1}{k+1}, \frac{1}{k}] $ we obtain the well-known Gauss map. For each $n$, we define an iterated function system $S_n$ (IFS) consisting of the maps $g_k$, $k  = 1, \dots , n$. Let $J_n$ be the Julia set (limit set) generated by $S_n$. In 1929 V. Jarnik \cite{jarnik1} estimated, using elementary methods, the rate of convergence of the Hausdorff dimension of the Julia set $J_n$ defined as the limit set of the IFS $S_n$. $J_n$ is a set of those irrational numbers in the set $[0,1]$, whose continued fraction expansion has entries bounded by $n$. In 1992 Doug Hensley \cite{Hensley} proved that $h_n$ has the following asymptotics:
\[
\lim \limits_{n \to \infty} n(1-h_n) = \frac{6}{\pi^2}.
\]
In 2016 Mariusz Urbański and Anna Zdunik in \cite{UZ} proved, using Hensley's result that for previously mentioned sets, we have continuity of the Hausdorff measure in Hausdorff dimension, meaning
\[
\lim \limits_{n \to \infty}H_{h_n}(J_n)  =  1,
\]
where $H_h$ - denotes the numerical value of Hausdorff measure in dimension $h$. This property is by no means obvious, as they also provided an example of Iterated Function System where this continuity does not hold.
\\
Given constants $a_n \geq 2$ for every $n \in \N$, each real number can be written in the form
\[
x = \frac{1}{a_1} + \frac{1}{a_1(a_1-1)a_2} + \dots \frac{1}{a_1(a_1-1)a_2 \dots a_{n-1}(a_{n-1}-1)a_n} + \dots
\]
This series expansion is called L\"uroth expansion, and was introduced in 1883 by L\"uroth \cite{luroth}. Each irrational number has a unique infinite expansion of this form and each rational number has either a finite expansion or a periodic one. The linear analogue of Gauss map is referred to as Generalized L\"uroth map. In recent years, the L\"uroth expansion has been of interest in many research papers. For example Barreira and Iommi in \cite{barreira-iommi} study Hausdorff dimension of a class of sets defined in terms of the frequencies of digits in the expansion. In 2013 Mance and Tseng show in \cite{mance-tseng} that the set of numbers with bounded Lüroth expansions (or bounded Lüroth series) is winning and strong winning. 
\par %In 2020 Tan and Zhang in \cite{tan-zhang} investigate the rate of growth of digits in the L\"uroth expansion of an irrational number, relative to the rate of approximation of the number by its convergents. The Hausdorff dimension of exceptional sets of points with a given relative growth rate is established. 
In 2021, Tan and Zhou in \cite{tan-zhou} studied the Jarník-like set of real numbers which can be well approximated by infinitely many of their convergents in the L\"uroth expansion:
\[
W(\psi) = \{x \in [0,1): |xq_n(x) - p_n(x)| < \psi(x) \text{ for infinitely many } n\in \N \}
\]
where $\psi: \mathbb{R} \to (0,\frac{1}{2}]$ is a positive function. They completely determine the Hausdorff dimension of $W(\psi)$. 
\\
\\
The main objective of this paper is finding conditions on the iterated function system that guarantee continuity of the Hausdorff measure. We manage to prove continuity of the Hausdorff measure in adequate dimension for sets generated by IFS consisting of the following linear decreasing functions
\[
f_k : \left[b_k, b_{k-1} \right] \to \left[0,1\right]
\]
such that
\[
f_k(b_{k-1}) = 0 \ {\rm and} \ f_k(b_{k}) = 1,
\]
where the sequence $b_k$, $k \in \{0,1,2,\dots\}$ monotonically decreases to $0$ as $k \to \infty$ and $b_0 = 1$. This is a generalization of the L\"uroth expansion.
By $g_k:[0,1]\to[b_k,b_{k-1}]$ we will denote the inverse map $f_k^{-1}$ for $k = 1, 2, \dots$
\par In addition, we assume that IFSes fulfill the following two conditions
\begin{CEquation}\label{warunek2}
\lim\limits_{n \to \infty} \left(1-h_n\right) \ln{n} = 0 ,
\end{CEquation}
where $h_n$ is the Hausdorff dimension of the limit set of the IFS  generated by $n$ initial maps $f_k$, $k = 1, \dots, n$ and 
\begin{CEquation}\label{warunek3}
\sup \limits_{k\in \mathbb{N}} \left \{\frac{b_k-b_{k+1}}{b_{k+1}} \right \} < \infty.
\end{CEquation}
By defining our sets in this way, we can examine the asymptotic of the Hausdorff dimension of Julia set $J_n$ generated by the first $n$ functions $g_k$ and based on this we can deduce the continuity of the Hausdorff measure. 
\\
Our main result of the first part is the following \Cref{razem}, stating
\begin{customthm}{\ref{razem}}
Let $S_n$ be iterated function system (IFS) defined in Definition \ref{IFS Sn}, which fulfills conditions \eqref{warunek2} and \eqref{warunek3}. Then
\[
\lim\limits_{n\to \infty} H_{h_n}(J_n) = 1,
\]
where $J_n$ is the limit set of the IFS $S_n$, $H_h$ denotes Hausdorff measure in Hausdorff dimension $h$, and $h_n$ is the Hausdorff dimension of the limit set $J_n$.
\end{customthm}
This result is a significant generalization of the result by M. Urbański and A. Zdunik in \cite{UZ}. Their counterexample to this property consists of linear transformations of the interval, however it does not fulfill Condition \eqref{warunek2}. In Section \ref{sec: examples} we discuss the assumption \eqref{warunek2} and provide a large collection of systems which satisfy it.
\par
In the second part of the paper, we analyze iterated function systems consisting of nonlinear functions. We assume, that $f_k(b_k) = 1$ and $f_k(b_{k-1}) = 0$, but we skip the linearity condition. $\kappa(g)$ denotes the distortion of $g$, defined in \Cref{def: distortion}. We assume that each function $g_k$, $k = 1,2, \dots $ is a $C^2$ function in some neighborhood of the interval $[0,1]$.  Now instead of previous conditions, we assume the following conditions:
\begin{CEquation}\label{warunek dodatkowy 3}
\lim\limits_{n \to \infty} \left(1-h_n\right) \ln{n} = 0,
\end{CEquation}
\begin{CEquation}\label{warunek dodatkowy 2}
\lim\limits_{n \to \infty}   \frac{b_k}{b_{k+1}} = 1 \text{ which is stronger than Condition \eqref{warunek3},}
\end{CEquation}
\begin{CEquation}\label{eq: warunek 4}
\bigg | \frac{g''_n(x)}{g_n'(x)} \bigg | < c \text{ for some } c \in (0,\infty),
\end{CEquation}
\begin{CEquation}\label{eq: warunek 5}
\lim \limits_{n \to \infty} \kappa(g_n) = 1
\end{CEquation}
and
\begin{CEquation}\label{eq: warunek 6}
|g'_n| < \alpha < 1 .
\end{CEquation}
for all $n \in \{1,2,\dots\}$.
Note that the last three conditions provide some bound on non-linearity of the system. Assuming those additional conditions, we prove in \Cref{thm: limit of the measure} the following result.
\begin{customthm}{\ref{thm: limit of the measure}}
Let $J_n$ be the limit set of the IFS $S_n$ fulfilling the conditions \eqref{warunek dodatkowy 3} -  \eqref{eq: warunek 6}. Then
\[
\lim \limits_{n \to \infty} H_{h_n}(J_n) = 1,
\]
where $h_n$ is the Hausdorff dimension of $J_n$ and $H_h$ is the Hausdorff measure in dimension $h$.
\end{customthm}
As one can see, usefulness of this theorem is mostly dependent on checking the Condition \eqref{warunek dodatkowy 3}. We show how to check this condition in non-linear case in Section \ref{section: nonlinear examples}.
\section{Notation}
We will start by introducing key definitions, theorems and notation.
\\
Let $f_k(x):[b_k,b_{k-1}]\to [0,1]$ for $k \in \mathbb{N}$ be a decreasing function, such that 
\[
f_k(b_{k}) = 1  \text{ and } f_k(b_{k-1}) = 0.
\]
By $g_k$ we will denote the inverse map $f_k^{-1}$, $k \in \N$. So $g_k$ maps the interval $[0,1]$ onto $[b_k, b_{k-1}]$. We also define the function $f :(0,1] \to [0,1]$ in the following way $f(x) = f_k(x)$ for $x \in (b_k,b_{k-1}]$,  $k = 1, 2, \dots$
\begin{dfn}\label{IFS Sn}
Iterated function system (IFS) $S_n$ is defined by limiting the collection of functions $g_k$ to first $n$, meaning $S_n = \{g_k\}_{k=1}^{n}$. 
\end{dfn}
Since the sequence $b_k$ defines functions $f_k$ and thus IFS $S_n$, we will say that $S_n$ is generated by $b_k$. 

\begin{notation}\label{not: limit set}
By $J_n$ we will denote the limit set created by the IFS $S_n$ 
\[
J_n = \bigcap\limits_{l=1}^\infty \bigcup\limits_{q_1, q_2 \dots q_l \in \{1, 2, \dots  n \}^{l} } g_{q_1}\circ g_{q_2} \circ \dots \circ g_{q_l}([0,1]).
\]
\end{notation}
\begin{notation}\label{hausdorff dimension}
We will denote Hausdorff dimension of the set $J_n$ by $h_n$ and Hausdorff measure of the set A in dimension $h$ by $H_{h}(A)$. 
\\We denote by ${\rm diam}(F)$ the diameter of the set F. We will also use the notation $|F|$ to denote the diameter of the set $F$.
\end{notation}
Next definition is widely-known condition known as Open Set Condition. One can find it in, for example, \cite{Falconer}.
\begin{dfn}\label{defi: open set condition}
We say that IFS composed of contractions $\{\phi_i\}_{i = 1}^n$ on $\mathbb{R}^n$ fulfills Open Set Condition (OSC), if there exists open set V such that following two conditions hold
\begin{equation}
    \bigcup\limits_{i=1}^n \phi_i(V) \subseteq V
\end{equation}
and the sets $\phi_i(V)$ are pairwise disjoint.
\end{dfn}
Now, for IFS which consists of linear functions, we have the following observation. Based on the fact that each IFS $S_n$ satisfies the Open Set Condition, we know that $h_n$ is a unique solution to the following equation 
\[
\sum\limits_{k=1}^{n} \left| b_{k} - b_{k-1}\right|^{h_n} = 1
\]
provided the functions generating the IFS are linear.
Proof of this classical fact can be found for example in \cite{Falconer}.
It follows from the above equation equation, that $\lim\limits_{n \to \infty} h_n = 1$ and $0<h_n<h_{n+1}<1$.
\begin{notation}
Let $\omega \in \mathbb{N}^k$. We will use the notation $\omega = [\omega_1, \omega_2, \omega_3, \dots, \omega_k]$. Then, we define $g_\omega$ as follows
\[
g_\omega = g_{\omega_1}\circ g_{\omega_2} \circ \dots \circ g_{\omega_k}.
\]
\end{notation}
\begin{dfn}\label{generacje fnl}
Let $S_n$ be the IFS generated by $b_k$, $k = 0,1, \dots, n-1$. We denote by $\mathcal{F}^{n}_l$ the $l$-th generation of intervals generated by $S_n$:
\[
  \mathcal{F}^n_l = \left \{ g_{i_1}\circ g_{i_2}\circ \dots \circ g_{i_l}([0,1]): i_1, i_2, \dots,i_l \in \{1,2,\dots, n\}  \right \}.
\]
\end{dfn}
Theorem 9.3 in \cite{Falconer} tells that the limit set of iterated function systems consisting of contracting similarities and fulfilling the Open Set Condition has positive and finite Hausdorff measure. Thus, we can define a normalized Hausdorff measure.
\begin{dfn}\label{defi: miara unormowana}
Let $X\subset \mathbb{R}$. If $0<H_{h_n}(X)<\infty$ then by $m_n$ we will denote the normalized $H_{h_n}$ Hausdorff measure on $X$
\[
m_n(A) := \frac{H_{h_n}(A \cap X)}{H_{h_n}(X)}
\]
\end{dfn}
\begin{prop}
Assume $S_n$ consists of linear functions. Let $m_n$ be the normalized Hausdorff measure on the set $J_n$. Then, for every Borel set $A \subset [0,1]$
\[
m_n(g_k(A)) = |g_k'|^{h_n} \cdot m_n(A)
\]
\end{prop}
As an immediate consequence of this fact, we get the following Corollary.
\begin{corollary}\label{cor: wlasnosc niezmienniczosci mn}
Assume $S_n$ consists of linear functions. Let A be a Borel set such that $A \subseteq [0,1]$. Then
\[
\frac{m_n(A)}{({\rm diam}A)^{h_n}} = \frac{m_n(g_k(A))}{({\rm diam}(g_k(A)))^{h_n}}
\]
for every $k \in \{1, 2, \dots, n \}$.
\end{corollary}

\begin{dfn}\label{density}
Let $J$ be a subset of the interval $[0,1]$. The $n-th$ density of the interval $J$ with respect to the measure $m_n$ denoted by $d_n(J)$ is given by the quotient $ d_n(J) := \frac{m_n(J)}{{\rm diam }(J)^{h_n}}$ where $h_n$ is Hausdorff dimension of the set $J_n$, and $m_n$ is normalized Hausdorff measure.
\end{dfn}
% \begin{notation}\label{not: density}
% By $n$-th density of an interval $F \subseteq [0,1]$ we will denote the following quotient
% \[
% d_n(F) := \frac{m_n(F)}{|F|^{h_n}}.
% \]
% \end{notation}
The main theorems used to prove our results are density theorems for the Hausdorff measure, see  \cite{Matilla} Theorem 6.2.
\begin{thm}
Let $X$ be a metric space, with Hausdorff dimension equal to $h$, such that the Hausdorff measure of $X$ in Hausdorff dimension $h$ is finite. Then
\[
\lim\limits_{r \to 0} \left( \sup \left\{  \frac{H_h(F)}{{\rm diam }^{h}(F)}:x \in F, \overline{F} = F, {\rm diam }(F) \leq r \right\} \right )= 1
\]
for $H_h$ - almost all  $x \in X$.
\end{thm}
From this theorem we get
\begin{thm}
Let $X$ be a metric space and by $H_h$ denote the Hausdorff measure in dimension $h$. Assume that $0 < H_h(X) < +\infty$ and let $H_h^1$ denote the normalized Hausdorff measure on $X$ in dimension $h$. Then
\[
H_h(X) = \lim\limits_{r\to 0} \left (\inf \left\{  \frac{{\rm diam }^{h}(F)}{H^1_h(F)}:x \in F, F \subset X, \overline{F} = F, {\rm diam }(F) \leq r \right\} \right)
\]
for $H_h^1$-almost all  $x\in X$.
\end{thm}
As a consequence of this Theorem and the fact, that in all Euclidean metric spaces the diameter of the closed convex hull of every set A is the same as the diameter of A, we get the following theorem for the subset of real line.
\begin{thm}\label{gestosc}
Let $X$ be a subset of an interval $\Delta \subset \mathbb{R}$ with finite and positive Hausdorff measure $H_h(X)$. Let $H_h^1$ be normalized Hausdorff measure on $X$. Then for $H_h^1$-almost all $x\in X$ we get
\[
H_h(X) = \lim\limits_{r\to 0} \inf \left\{  \frac{{\rm diam }^{h}(F)}{H^1_h(F\cap X)}:x \in F, F\subset \Delta \ is \ a \ closed \ interval, \ and \ {\rm diam }(F) \leq r \right\}.
\]
\end{thm}
We can take this theorem one step further while talking about iterated function systems consisting of linear functions. The following theorem holds true.
\begin{thm}\label{gestosc lepsza}
Let $\Delta \subset \mathbb{R}$ be an interval. Let $X \subset \Delta$ be the limit set of the iterated function system consisting of contracting similarities $g_j$ such that $g_j(\Delta) \subset \Delta$ and $0 < H_h(X) < +\infty$, then
\[
H_h(X) = \inf \left\{\frac{{\rm diam }^{h}(F)}{H^1_h(F\cap X)}: F \ is \ a\ closed \ interval \right\}.
\]
\end{thm}
Proof of this theorem can be found in Theorem 5.1 in \cite{UZ}.
\par We shall use the following straightforward consequence of Jensen's inequality.
\begin{lemma}\label{lem: oszacowanie na sume wi}
Let $w_1, w_2, \dots, w_k$ be positive real numbers such that $   \sum\limits_{i = 1}^k w_i = 1.$ Then for all $0<a<1$ the following holds $\sum\limits_{i = 1}^k \left ( w_i \right)^{a} \leq k^{1-a}$.

\end{lemma}
\part{Linear systems}
In this part we are focusing on iterated function systems that consist of affine functions. To be precise, we assume that $f_k$ which form our IFS $S_n$ are affine.  
\section{Estimate from above}
In this section, we shall prove that the upper limit of the Hausdorff measure of the limit set generated by the first $n$ functions is less than or equal to 1. This is the easy part of the proof. 
\begin{thm}\label{gora}
Let $S_n$ be an IFS consiting of linear functions fulfilling conditions \eqref{warunek2} and \eqref{warunek3}. Then
\[
\limsup\limits_{n \to \infty} H_{h_n}(J_n) \leq 1.
\]
\end{thm}
\begin{proof}
Fix $n$. We will show that
\[
H_{h_n}(J_n) \leq 1
\]
for each $n\in \mathbb{N}$. We will show a sequence of covers of the set $J_n$ by sets from the collection $\mathcal{F}_{k}^n$ (defined in \ref{generacje fnl}) which are k-th generation cylinders obtained by iterating first $n$ functions defined in Definition \ref{generacje fnl} and we will justify that for every $k \in \mathbb{N}$ the following holds:
\[
\sum\limits_{F \in \mathcal{F}_{k}^n} \left| F \right|^{h_n} = 1.
\]
Note that for each $n$ we have that
\[
\max\limits_{F\in \mathcal{F}_k^n}\left( \left| F \right|\right) \xrightarrow{k\to \infty} 0.
\]
Thus, by the definition of the Hausdorff measure we shall conclude that $H_{h_n}(J_n) \leq 1$, and, therefore
\[
\limsup\limits_{n \to \infty} H_{h_n}(J_n) \leq 1.
\]
Indeed,
\[
 \sum\limits_{F \in \mathcal{F}_{k}^n} \left| F \right|^{h_n} = \sum\limits_{j_1\dots j_k}(a_{j_1}\dots a_{j_k})^{h_n}
\]
where $a_j = b_{j-1} - b_{j}$ and the sum runs over all sequences  $(j_1 \dots j_k)  \in \{ 1,2, \dots, n\}^k$. 
\\
Next,
\[
 \sum\limits_{j_1\dots j_k}(a_{j_1}\dots a_{j_k})^{h_n}   = (a_1^{h_n} + a_2^{h_n} + \dots + a_{n}^{h_n})^{k} 
\]
\[ 
 = ((a_1^{h_n} + a_2^{h_n} + \dots + a_{n}^{h_n}))^k  = 1,
\]
which ends the proof.
\end{proof}
Note that this theorem does not require conditions \eqref{warunek2} nor \eqref{warunek3}.
\section{Estimate from below}
In this section we will focus on showing that the Hausdorff measure of the limit set $J_n$ obtained from linear IFS $S_n$ is continuous, which means
\[
\lim\limits_{n\to \infty} H_{h_n}(J_n) = 1
\]
as long as condition
\begin{equation}\tag{\ref{warunek2}}
\lim\limits_{n \to \infty} \left(1-h_n\right) \ln{n} = 0 
\end{equation}
and
\begin{equation}\tag{\ref{warunek3}}
\sup \limits_{k\in \mathbb{N}} \left \{\frac{b_k-b_{k+1}}{b_{k+1}} \right \} < \infty
\end{equation}
are met. 
\\
Our strategy for proving the estimate of the Hausdorff measure from above is to prove that the upper limit of the densities $d_n$ of all intervals contained in $[0,1]$ is at most $1$. 

\subsection{Strategy of the proof}
The proof is split into three main parts. The first one is proving that the lower limit of the densities on the intervals of the form $[0,r]$ is greater or equal to 1. The second one is proving the same on the intervals of the form $[b_{k+l}, b_{k}]$. The final part is putting those theorems together and proving step by step that
\begin{equation}\label{eq: linear lower bound}
    \varliminf\limits_{n\to \infty} H_{h_n}(J_n) \geq 1.
\end{equation}

\subsection{Estimate on the intervals $[0,r]$}

Now, we will focus on preliminary lemmas required to estimate the density on the intervals of the form $[0,r]$. The special case where instead of Condition \eqref{warunek3}, we assume a much stronger condition:
\begin{equation}\tag{\ref{warunek dodatkowy 2}}
    \lim \limits_{k \to \infty} \frac{b_{k-1}}{b_{k}} = 1 
\end{equation}
or, equivalently, 
\[
    \lim \limits_{k \to \infty}\frac{b_{k-1}-b_{k}}{b_{k}}=0
\]
can be simplified and conducted using similar ideas to the proof found in \cite{UZ}.
\begin{lemma}\label{b przez a}
Let  $\mathcal{F}_l^n$ be the set of all intervals of $l$-th generation of the IFS as defined in \Cref{generacje fnl}, generated by $S_n = \{g_1 \dots g_n\}$. Furthermore, let $b_k$ satisfy the following 
\[
\sup \limits_{k\in \mathbb{N}} \left \{\frac{b_{k-1}-b_{k}}{b_{k}} \right \} < \infty.
\]
Then for any $\varepsilon >0$ there exists $l_0$ such that for every $l>l_0$, $n \in \mathbb{N}$ and for every interval $[a,b] \in \mathcal{F}_l^n$ the following holds:
\[
\frac{b}{a} \leq 1 + \varepsilon.
\]
\end{lemma}
\begin{proof}
    
Let $[a,b] \in \mathcal{F}_l^n$. Recall that $a_k$ = $|b_k-b_{k-1}|$. Then the following holds 
\[
b-a = a_{i_1}\cdot a_{i_2}\dots a_{i_l},
\]
because $[a,b] = g_{i_1}\circ g_{i_2} \circ \dots \circ g_{i_l}([0,1])$ for some sequence of maps $g_{i_1}, g_{i_2}, \dots ,g_{i_l}$, $i_1,\dots i_l \leq n$ . Now, choose some $\alpha \in (0,1)$ such that $0 < a_j \leq \alpha $ for all $j \in \mathbb{N}$ and consider the quotient
\begin{multline*}
\frac{b}{a} = \frac{a + a_{i_1}\cdot a_{i_2}\dots a_{i_l}}{a} = 1 + \frac{a_{i_1}\cdot a_{i_2}\dots a_{i_l}}{a} \leq
\\
\leq  1 + \sup \limits_{k\in \mathbb{N}} \left \{\frac{b_{k-1}-b_{k}}{b_{k}} \right \} \cdot a_{i_2}\dots a_{i_l} \leq  1 +\sup \limits_{k\in \mathbb{N}} \left \{\frac{b_{k-1}-b_{k}}{b_{k}} \right \} \cdot \alpha^{l-1} \leq 1 + \varepsilon.
\end{multline*}
where the first inequality follows from the fact that $[a,b] \subset [b_{i_1},b_{i_1-1}]$ and thus $a \geq b_{i_1}$. The final inequality holds for sufficiently large $l$.
\end{proof}

\begin{prop}\label{dziobanie}
If IFS fulfills the conditions
\eqref{warunek2} and \eqref{warunek3},
then
\[
\limsup \limits_{n \to \infty} \left ( \sup \left \{ \frac{m_n([0,r])}{r^{h_n}}: r \in (0,1)\right \} \right ) \leq 1.
\]
\end{prop}
\begin{proof}
Fix $\varepsilon >0$. Let $n \in \mathbb{N}$. Take an arbitrary $r \in (0,1)$. If $r \leq b_{n}$, then $m_n([0,r]) = 0$. So, let $r > b_{n}$. From  \Cref{b przez a}, we can find odd $l \in \mathbb{N}$ large enough such that $\frac{b}{a} \leq 1 + \frac{\varepsilon}{4}$ for all intervals $[a,b] \in \cup_{n \in \mathbb{N}} \mathcal{F}_l^n$. The integer $l$ is chosen to be odd in order to guarantee the form of the last component in the quotient \eqref{eq: measure div r}, which will appear later on, because the maps in our IFS reverse the orientation. 
\par If $r \not \in \bigcup\limits_{F \in \mathcal{F}^n_l} F$ (i.e. if the point r is in some "gap" of level $l$ of the Cantor set $J_n$) then we can replace $r$ with the closest $b$, such that $b<r$ and $[a,b]$ is in $l$-th generation cylinder, without changing measure $m_n$ of the interval $(0,r)$ (see Figure \ref{fig:r to left}), meaning $m_n((0,r)) = m_n((0,b])$.  We also get that ${\rm diam}((0,r)) > {\rm diam}((0,b])$, hence $\frac{m_n((0,r))}{{\rm diam}((0,r))} \leq \frac{m_n((0,b])}{{\rm diam}((0,b])}$.
\\

\begin{figure}
    \centering

\tikzset{every picture/.style={line width=0.75pt}} %set default line width to 0.75pt        
%\begin{figure}

\begin{tikzpicture}[x=0.75pt,y=0.75pt,yscale=-1,xscale=1] 
%uncomment if require: \path (0,237); %set diagram left start at 0, and has height of 237

%Straight Lines [id:da08879702533813338] 
\draw    (101,50.5) -- (247,50.5) ;
%Shape: Parabola [id:dp2605190998682929] 
\draw   (239.96,66.78) .. controls (249.16,56.01) and (249.35,45.08) .. (240.51,33.99) ;
%Shape: Parabola [id:dp22762226251705364] 
\draw   (107.38,33.95) .. controls (98.62,45.09) and (98.87,56.02) .. (108.14,66.74) ;
%Straight Lines [id:da7948330345041741] 
\draw    (252,50.5) -- (351,50.5) ;
%Shape: Parabola [id:dp0764798496895398] 
\draw   (346.23,66.78) .. controls (352.47,56) and (352.59,45.07) .. (346.6,33.99) ;
%Shape: Parabola [id:dp5163398220915363] 
\draw   (256.33,33.95) .. controls (250.38,45.09) and (250.55,56.02) .. (256.84,66.73) ;
%Straight Lines [id:da677656827591004] 
\draw    (399,50.5) -- (467,50.5) ;
%Shape: Parabola [id:dp4412976336242531] 
\draw   (463.72,66.78) .. controls (468,56) and (468.09,45.07) .. (463.98,33.99) ;
%Shape: Parabola [id:dp5407085828074201] 
\draw   (401.97,33.95) .. controls (397.89,45.09) and (398,56.01) .. (402.32,66.73) ;
%Straight Lines [id:da7955569577499852] 
\draw    (489,50.5) -- (557,50.5) ;
%Shape: Parabola [id:dp3472467628796776] 
\draw   (553.72,66.78) .. controls (558.01,56) and (558.09,45.07) .. (553.98,33.99) ;
%Shape: Parabola [id:dp7638478655270975] 
\draw   (491.97,33.95) .. controls (487.89,45.09) and (488,56.01) .. (492.32,66.73) ;
%Straight Lines [id:da2747734276218772] 
\draw [color={rgb, 255:red, 208; green, 2; blue, 27 }  ,draw opacity=1 ]   (372,42) -- (372,58.5) ;
%Straight Lines [id:da1958851974438438] 
\draw    (101,191.5) -- (247,191.5) ;
%Shape: Parabola [id:dp08690770558366645] 
\draw   (239.96,207.78) .. controls (249.16,197.01) and (249.35,186.08) .. (240.51,174.99) ;
%Shape: Parabola [id:dp5660731442915539] 
\draw   (107.38,174.95) .. controls (98.62,186.09) and (98.87,197.02) .. (108.14,207.74) ;
%Straight Lines [id:da7219893028990825] 
\draw    (252,191.5) -- (351,191.5) ;
%Shape: Parabola [id:dp24753035203379903] 
\draw   (346.23,207.78) .. controls (352.47,197) and (352.59,186.07) .. (346.6,174.99) ;
%Shape: Parabola [id:dp17656954132239933] 
\draw   (256.33,174.95) .. controls (250.38,186.09) and (250.55,197.02) .. (256.84,207.73) ;
%Straight Lines [id:da5001024689612548] 
\draw    (399,191.5) -- (467,191.5) ;
%Shape: Parabola [id:dp051297858394857565] 
\draw   (463.72,207.78) .. controls (468,197) and (468.09,186.07) .. (463.98,174.99) ;
%Shape: Parabola [id:dp7016037994679076] 
\draw   (401.97,174.95) .. controls (397.89,186.09) and (398,197.01) .. (402.32,207.73) ;
%Straight Lines [id:da9537794018880308] 
\draw    (489,191.5) -- (557,191.5) ;
%Shape: Parabola [id:dp17655687254131114] 
\draw   (553.72,207.78) .. controls (558.01,197) and (558.09,186.07) .. (553.98,174.99) ;
%Shape: Parabola [id:dp43338343124637047] 
\draw   (491.97,174.95) .. controls (487.89,186.09) and (488,197.01) .. (492.32,207.73) ;
%Straight Lines [id:da34244967457569286] 
\draw [color={rgb, 255:red, 208; green, 2; blue, 27 }  ,draw opacity=1 ]   (351,183) -- (351,199.5) ;
%Straight Lines [id:da41813252793078126] 
\draw [color={rgb, 255:red, 208; green, 2; blue, 27 }  ,draw opacity=1 ]   (374,95) -- (349,95) ;
\draw [shift={(347,95)}, rotate = 360] [color={rgb, 255:red, 208; green, 2; blue, 27 }  ,draw opacity=1 ][line width=0.75]    (10.93,-3.29) .. controls (6.95,-1.4) and (3.31,-0.3) .. (0,0) .. controls (3.31,0.3) and (6.95,1.4) .. (10.93,3.29)   ;
%Straight Lines [id:da5618527842260237] 
\draw    (372,103) -- (372,155.99) ;
\draw [shift={(372,157.99)}, rotate = 270] [color={rgb, 255:red, 0; green, 0; blue, 0 }  ][line width=0.75]    (10.93,-3.29) .. controls (6.95,-1.4) and (3.31,-0.3) .. (0,0) .. controls (3.31,0.3) and (6.95,1.4) .. (10.93,3.29)   ;

% Text Node
\draw (368,61) node [anchor=north west][inner sep=0.75pt]   [align=left] {\textcolor[rgb]{0.82,0.01,0.11}{r}};
% Text Node
\draw (569,40.4) node [anchor=north west][inner sep=0.75pt]    {$\mathcal{F}_{l}^{n}$};
% Text Node
\draw (569,181.4) node [anchor=north west][inner sep=0.75pt]    {$\mathcal{F}_{l}^{n}$};
% Text Node
\draw (249,70) node [anchor=north west][inner sep=0.75pt]   [align=left] {a};
% Text Node
\draw (348.23,71.78) node [anchor=north west][inner sep=0.75pt]   [align=left] {b};
% Text Node
\draw (248,211) node [anchor=north west][inner sep=0.75pt]   [align=left] {a};
% Text Node
\draw (348.23,212.78) node [anchor=north west][inner sep=0.75pt]  [color={rgb, 255:red, 208; green, 2; blue, 27 }  ,opacity=1 ] [align=left] {b };

\end{tikzpicture}
    \caption{Moving $r$ to the left without changing the value of  the measure $m_n([0,r])$}
    \label{fig:r to left}
\end{figure}
%\caption{Replacing $r$ with closest endpoint of the generation}\label{fig: przesuwanie r}
%\end{figure}
\par Thus, from now on we can assume, that $r \in F$ for some $F \in \mathcal{F}^n_l$. By $b_{q_1, q_2, \dots q_l-1}$ we denote the right endpoint of the interval of the $l$-th generation $\mathcal{F}_l^n$ such that
\[
b_{q_1, q_2, \dots q_l-1} = g_{q_1} \circ  g_{q_{2}} \circ \dots \circ g_{q_{l-1}} \circ g_{q_{l}}(0).
\]
Then there exists a unique sequence of numbers $q_1, q_2, \dots q_l \in \{1,2, \dots, n\}$, such that $r \in g_{q_1} \circ g_{q_2} \circ \dots \circ g_{q_l}([0,1])$, and hence
\begin{equation}\label{eq: r w kleszcze}    
b_{q_1, q_2, \dots q_l } <   r \leq   b_{q_1, q_2, \dots q_l-1}.
\end{equation}
Inequality \eqref{eq: r w kleszcze} holds this way due to fact that $l$ is odd - see Figure \ref{fig:ab lth generation}.

\begin{figure}
    \centering

\tikzset{every picture/.style={line width=0.75pt}} %set default line width to 0.75pt        

\begin{tikzpicture}[x=0.75pt,y=0.75pt,yscale=-1,xscale=1]
%uncomment if require: \path (0,357); %set diagram left start at 0, and has height of 357

%Straight Lines [id:da6873476189943658] 
\draw    (58,29) -- (412,29) ;
%Straight Lines [id:da2720945758685225] 
\draw    (58,38) -- (58,19) ;
%Straight Lines [id:da31744715094552145] 
\draw    (411,39) -- (411,20) ;
%Straight Lines [id:da818523914709045] 
\draw    (368,39) -- (368,20) ;
%Straight Lines [id:da4194941258585271] 
\draw    (299,39) -- (299,20) ;
%Straight Lines [id:da8489178199740147] 
\draw [color={rgb, 255:red, 208; green, 2; blue, 27 }  ,draw opacity=1 ]   (341,39) -- (341,20) ;
%Straight Lines [id:da40274983335394154] 
\draw    (182,39) -- (182,20) ;
%Straight Lines [id:da4683456975453787] 
\draw    (141,39) -- (141,20) ;
%Straight Lines [id:da5084222026342482] 
\draw    (299,39) -- (61.84,140.21) ;
\draw [shift={(60,141)}, rotate = 336.89] [color={rgb, 255:red, 0; green, 0; blue, 0 }  ][line width=0.75]    (10.93,-3.29) .. controls (6.95,-1.4) and (3.31,-0.3) .. (0,0) .. controls (3.31,0.3) and (6.95,1.4) .. (10.93,3.29)   ;
%Straight Lines [id:da6362803354632833] 
\draw    (60,151) -- (414,151) ;
%Straight Lines [id:da8098592041677858] 
\draw    (60,160) -- (60,141) ;
%Straight Lines [id:da22946438775785816] 
\draw    (413,161) -- (413,142) ;
%Straight Lines [id:da5663018701208815] 
\draw    (301,161) -- (301,142) ;
%Straight Lines [id:da16247375431242028] 
\draw [color={rgb, 255:red, 208; green, 2; blue, 27 }  ,draw opacity=1 ]   (231,161) -- (231,142) ;
%Straight Lines [id:da8864192345329098] 
\draw    (184,161) -- (184,142) ;
%Straight Lines [id:da8521596903324306] 
\draw    (109,161) -- (109,142) ;
%Curve Lines [id:da5716179241727477] 
\draw    (304,199) .. controls (303,229) and (242,195) .. (241,229) ;
%Curve Lines [id:da0008730429157801645] 
\draw    (183,200) .. controls (182,230) and (242,195) .. (241,229) ;
%Straight Lines [id:da7445558540072331] 
\draw    (241,229) -- (241,247) ;
\draw [shift={(241,249)}, rotate = 270] [color={rgb, 255:red, 0; green, 0; blue, 0 }  ][line width=0.75]    (10.93,-3.29) .. controls (6.95,-1.4) and (3.31,-0.3) .. (0,0) .. controls (3.31,0.3) and (6.95,1.4) .. (10.93,3.29)   ;
%Straight Lines [id:da6115623082168621] 
\draw    (57,305) -- (411,305) ;
%Straight Lines [id:da3903721397090182] 
\draw    (57,314) -- (57,295) ;
%Straight Lines [id:da5088469419474867] 
\draw    (410,315) -- (410,296) ;
%Straight Lines [id:da2766442782361439] 
\draw    (298,315) -- (298,296) ;
%Straight Lines [id:da22192529874127553] 
\draw [color={rgb, 255:red, 208; green, 2; blue, 27 }  ,draw opacity=1 ]   (237,315) -- (237,296) ;
%Straight Lines [id:da6045856810292435] 
\draw    (181,315) -- (181,296) ;
%Straight Lines [id:da036501042062925704] 
\draw    (368,39) -- (412.2,140.17) ;
\draw [shift={(413,142)}, rotate = 246.4] [color={rgb, 255:red, 0; green, 0; blue, 0 }  ][line width=0.75]    (10.93,-3.29) .. controls (6.95,-1.4) and (3.31,-0.3) .. (0,0) .. controls (3.31,0.3) and (6.95,1.4) .. (10.93,3.29)   ;
%Straight Lines [id:da15599331105385172] 
\draw    (185,250) -- (58.84,304.21) ;
\draw [shift={(57,305)}, rotate = 336.75] [color={rgb, 255:red, 0; green, 0; blue, 0 }  ][line width=0.75]    (10.93,-3.29) .. controls (6.95,-1.4) and (3.31,-0.3) .. (0,0) .. controls (3.31,0.3) and (6.95,1.4) .. (10.93,3.29)   ;
%Straight Lines [id:da05656395730874353] 
\draw    (310,253) -- (409.22,304.08) ;
\draw [shift={(411,305)}, rotate = 207.24] [color={rgb, 255:red, 0; green, 0; blue, 0 }  ][line width=0.75]    (10.93,-3.29) .. controls (6.95,-1.4) and (3.31,-0.3) .. (0,0) .. controls (3.31,0.3) and (6.95,1.4) .. (10.93,3.29)   ;

% Text Node
\draw (53,45) node [anchor=north west][inner sep=0.75pt]   [align=left] {0};
% Text Node
\draw (406,46) node [anchor=north west][inner sep=0.75pt]   [align=left] {1};
% Text Node
\draw (337,46) node [anchor=north west][inner sep=0.75pt]  [color={rgb, 255:red, 208; green, 2; blue, 27 }  ,opacity=1 ] [align=left] {r};
% Text Node
\draw (344,44) node [anchor=north west][inner sep=0.75pt]   [align=left] {$\displaystyle b_{q_{1}-1}$};
% Text Node
\draw (289,44) node [anchor=north west][inner sep=0.75pt]   [align=left] {$\displaystyle b_{q_{1}}$};
% Text Node
\draw (232,42) node [anchor=north west][inner sep=0.75pt]   [align=left] {...};
% Text Node
\draw (173,47) node [anchor=north west][inner sep=0.75pt]   [align=left] {$\displaystyle b_{n}$};
% Text Node
\draw (135,46) node [anchor=north west][inner sep=0.75pt]   [align=left] {$\displaystyle b_{n+1}$};
% Text Node
\draw (227,166) node [anchor=north west][inner sep=0.75pt]  [color={rgb, 255:red, 208; green, 2; blue, 27 }  ,opacity=1 ] [align=left] {r};
% Text Node
\draw (397,168) node [anchor=north west][inner sep=0.75pt]   [align=left] {$\displaystyle b_{q_{1} -1}$};
% Text Node
\draw (51,168) node [anchor=north west][inner sep=0.75pt]   [align=left] {$\displaystyle b_{q_{1}}$};
% Text Node
\draw (144.78,166) node [anchor=north west][inner sep=0.75pt]   [align=left] {...};
% Text Node
\draw (175,168) node [anchor=north west][inner sep=0.75pt]   [align=left] {$\displaystyle b_{q_{1} ,q_{2}-1}$};
% Text Node
\draw (95.46,168) node [anchor=north west][inner sep=0.75pt]   [align=left] {$\displaystyle b_{q_{1} ,\ 1}$};
% Text Node
\draw (277,168) node [anchor=north west][inner sep=0.75pt]   [align=left] {$\displaystyle b_{q_{1} ,q_{2} }$};
% Text Node
\draw (435,22) node [anchor=north west][inner sep=0.75pt]   [align=left] {First generation};
% Text Node
\draw (435,144) node [anchor=north west][inner sep=0.75pt]   [align=left] {Second generation};
% Text Node
\draw (233,320) node [anchor=north west][inner sep=0.75pt]  [color={rgb, 255:red, 208; green, 2; blue, 27 }  ,opacity=1 ] [align=left] {r};
% Text Node
\draw (398,320) node [anchor=north west][inner sep=0.75pt]   [align=left] {$\displaystyle b_{q_{1} ,q_{2} ,...,q_{l-1}}$};
% Text Node
\draw (232.78,258) node [anchor=north west][inner sep=0.75pt]   [align=left] {...};
% Text Node
\draw (105.78,322) node [anchor=north west][inner sep=0.75pt]   [align=left] {...};
% Text Node
\draw (130,320) node [anchor=north west][inner sep=0.75pt]   [align=left] {$\displaystyle b_{q_{1} ,q_{2} ,...,q_{l-1} ,q_{l}}$};
% Text Node
\draw (255,320) node [anchor=north west][inner sep=0.75pt]   [align=left] {$\displaystyle b_{q_{1} ,q_{2} ,...,q_{l-1} ,q_{l}-1}$};
% Text Node
\draw (18,320) node [anchor=north west][inner sep=0.75pt]   [align=left] {$\displaystyle b_{q_{1} ,q_{2} ,...,q_{l-1}-1}$};
% Text Node
\draw (373.78,319) node [anchor=north west][inner sep=0.75pt]   [align=left] {...};
% Text Node
\draw (435,298) node [anchor=north west][inner sep=0.75pt]   [align=left] {$\displaystyle l$-th generation};
% Text Node
\draw (175,350) node [anchor=north west][inner sep=0.75pt]   [align=left] [rotate = 90] {$\displaystyle = $};
% Text Node
\draw (175,350) node [anchor=north west][inner sep=0.75pt]   [align=left] {$\displaystyle a $};
% Text Node

\draw (295,350) node [anchor=north west][inner sep=0.75pt]   [align=left] [rotate = 90] {$\displaystyle = $};
% Text Node
\draw (295,350) node [anchor=north west][inner sep=0.75pt]   [align=left] {$\displaystyle b $};
\end{tikzpicture}
    \caption{Interval $[a,b]$ as endpoints of $l$-th generation}
    \label{fig:ab lth generation}
\end{figure}
Let us focus on the estimate of the numerator $m_n([0,r])$, for which we estimate the measure of $[0,r]$ from above using the fact that $m_n([0,r]) \leq m_n([0,b_{q_1,q_2,\dots, q_{l-1}, q_{l}-1}])$. The measure of the interval $[0,b_{q_1,q_2,\dots, q_{l-1}, q_{l}-1}]$ can be expressed as the measure of the union of the following intervals $[0,b_{q_1}]$, $[b_{q_1}, b_{q_1,q_2}]$, $\dots$, $[ b_{q_1,q_2, \dots q_{l-1}},b_{q_1,q_2, \dots q_{l-1}, q_l-1}]$. The measure of the first interval is equal to
\[
m_n([0,b_{q_1}]) = \sum \limits_{j=q_1}^{n-1}(b_{j}-b_{j+1})^{h_n}.
\]
The measure of the interval  $[b_{q_1}, b_{q_1,q_2}]$ is equal to
\[
m_n([b_{q_1}, b_{q_1,q_2}]) = (b_{q_1-1}-b_{q_1})^{h_n} \sum \limits_{j=0}^{q_2-1}(b_{j}-b_{j+1})^{h_n}.
\]
Using induction, we have that the measure of the interval $[0,b_{q_1,q_2,\dots, q_{l-1}, q_{l}-1}]$ is as follows
\begin{equation}\label{eq: rozpisanie miary}    
m_n([0,r]) \leq m_n([0,b_{q_1,q_2,\dots, q_{l-1}, q_{l}-1}]) =
\end{equation}
\\
\\
\resizebox{\textwidth}{!}{$
= \sum \limits_{j=q_1}^{n-1}(b_{j}-b_{j+1})^{h_n} + (b_{q_1-1}-b_{q_1})^{h_n} \sum \limits_{j=0}^{q_2-1}(b_{j}-b_{j+1})^{h_n} + \dots + \prod \limits _{k=1}^{l-1}(b_{q_k-1}-b_{q_k})^{h_n}\sum  \limits_{j=q_l-1}^{n-1} (b_{j}-b_{j+1})^{h_n}.
$}
Note that the limits of the last summand are due to $l$ being odd. We thus have a sum of components, each of them raised to the power of $h_n$. We want to multiply and divide it by similar sum, but the whole sum will be raised to the power of $h_n$ instead of individual summands. This means that the expression in the right-hand side of formula \eqref{eq: rozpisanie miary} can be viewed as follows 
\\
\\
\resizebox{\textwidth}{!}{$
\frac{\sum \limits_{j=q_1}^{n-1}(b_{j}-b_{j+1})^{h_n} + (b_{q_1-1}-b_{q_1})^{h_n} \sum \limits_{j=0}^{q_2-1}(b_{j}-b_{j+1})^{h_n} + \dots + \prod \limits _{k=1}^{l-1}(b_{q_k-1}-b_{q_k})^{h_n}\sum  \limits_{j=q_l-1}^{n-1} (b_{j}-b_{j+1})^{h_n}}{\left[\sum \limits_{j=q_1}^{n-1}(b_{j}-b_{j+1}) + (b_{q_1-1}-b_{q_1}) \sum \limits_{j=0}^{q_2-1}(b_{j}-b_{j+1}) +\dots + \prod \limits _{k=1}^{l-1}(b_{q_k-1}-b_{q_k})\sum  \limits_{j=q_l-1}^{n-1} (b_{j}-b_{j+1})\right]^{h_n}} \cdot
$}
\\
\\
\resizebox{\textwidth}{!}{$
\cdot \left[\sum \limits_{j=q_1}^{n-1}(b_{j}-b_{j+1}) + (b_{q_1-1}-b_{q_1}) \sum \limits_{j=0}^{q_2-1}(b_{j}-b_{j+1}) + \dots + \prod \limits _{k=1}^{l-1}(b_{q_k-1}-b_{q_k})\sum  \limits_{j=q_l-1}^{n-1} (b_{j}-b_{j+1})\right]^{h_n}\leq
$}
\\
\\
\resizebox{\textwidth}{!}{$
\leq A \cdot \left[\sum \limits_{j=q_1}^{n-1}(b_{j}-b_{j+1}) + (b_{q_1-1}-b_{q_1}) \sum \limits_{j=0}^{q_2-1}(b_{j}-b_{j+1}) +\dots + \prod \limits _{k=1}^{l-1}(b_{q_k-1}-b_{q_k})\sum  \limits_{j=q_l-1}^{n-1} (b_{j}-b_{j+1})\right]^{h_n}
$}
\\
\\
\resizebox{\textwidth}{!}{$
\leq A \cdot \left[\sum \limits_{j=q_1}^{\infty}(b_{j}-b_{j+1}) + (b_{q_1-1}-b_{q_1}) \sum \limits_{j=0}^{q_2-1}(b_{j}-b_{j+1}) +\dots + \prod \limits _{k=1}^{l-1}(b_{q_k-1}-b_{q_k})\sum  \limits_{j=q_l-1}^{\infty} (b_{j}-b_{j+1})\right]^{h_n} = 
$}
\[
= A \cdot C,
\]
where we denoted by A the first quotient. In the last inequality, we replaced all of the sums running from $q_i$ to $n-1$ with the sums running from $q_i$ to infinity. Now we can estimate $r^{h_n}$ from below using the fact that $[0,r] \supseteq [0,b_{q_1,q_2,\dots, q_l+1}]$ and thus $r \geq b_{q_1,q_2,\dots, q_l+1}$.\\

\begin{figure}

\tikzset{every picture/.style={line width=0.75pt}} %set default line width to 0.75pt        

\begin{tikzpicture}[x=0.75pt,y=0.75pt,yscale=-1,xscale=1]
%uncomment if require: \path (0,200); %set diagram left start at 0, and has height of 200

%Straight Lines [id:da2838317847686882] 
\draw    (255,48) -- (511.23,48) ;
%Straight Lines [id:da10007893412247448] 
\draw    (255,88) -- (255,37) ;
%Straight Lines [id:da756324269356777] 
\draw [color={rgb, 255:red, 208; green, 2; blue, 27 }  ,draw opacity=1 ]   (435,88) -- (435,40) ;
%Straight Lines [id:da5769840883792106] 
\draw    (59,48) -- (175,48) -- (255,48) ;
%Straight Lines [id:da7788280316012637] 
\draw    (58,57) -- (58,38) ;
%Shape: Rectangle [id:dp9674972674685972] 
\draw   (58,23) -- (496,23) -- (496,48) -- (58,48) -- cycle ;
%Straight Lines [id:da5672092737645] 
\draw [color={rgb, 255:red, 144; green, 19; blue, 254 }  ,draw opacity=1 ]   (58,23) -- (83,48) ;
%Straight Lines [id:da025427405971197303] 
\draw [color={rgb, 255:red, 144; green, 19; blue, 254 }  ,draw opacity=1 ]   (78,23) -- (103,48) ;
%Straight Lines [id:da3076267330707755] 
\draw [color={rgb, 255:red, 144; green, 19; blue, 254 }  ,draw opacity=1 ]   (98,23) -- (123,48) ;
%Straight Lines [id:da8801818406515122] 
\draw [color={rgb, 255:red, 144; green, 19; blue, 254 }  ,draw opacity=1 ]   (118,23) -- (143,48) ;
%Straight Lines [id:da3831875923305521] 
\draw [color={rgb, 255:red, 144; green, 19; blue, 254 }  ,draw opacity=1 ]   (138,23) -- (163,48) ;
%Straight Lines [id:da2437825492900242] 
\draw [color={rgb, 255:red, 144; green, 19; blue, 254 }  ,draw opacity=1 ]   (158,23) -- (183,48) ;
%Straight Lines [id:da18072800488275542] 
\draw [color={rgb, 255:red, 144; green, 19; blue, 254 }  ,draw opacity=1 ]   (178,23) -- (203,48) ;
%Straight Lines [id:da2215075135155644] 
\draw [color={rgb, 255:red, 144; green, 19; blue, 254 }  ,draw opacity=1 ]   (198,23) -- (223,48) ;
%Straight Lines [id:da33716956973046075] 
\draw [color={rgb, 255:red, 144; green, 19; blue, 254 }  ,draw opacity=1 ]   (218,23) -- (243,48) ;
%Straight Lines [id:da7543989693206855] 
\draw [color={rgb, 255:red, 144; green, 19; blue, 254 }  ,draw opacity=1 ]   (238,22) -- (263,47) ;
%Straight Lines [id:da10605306537415082] 
\draw [color={rgb, 255:red, 144; green, 19; blue, 254 }  ,draw opacity=1 ]   (258,23) -- (283,48) ;
%Straight Lines [id:da6301913235074647] 
\draw [color={rgb, 255:red, 144; green, 19; blue, 254 }  ,draw opacity=1 ]   (278,23) -- (303,48) ;
%Straight Lines [id:da9441626867206627] 
\draw [color={rgb, 255:red, 144; green, 19; blue, 254 }  ,draw opacity=1 ]   (298,24) -- (323,49) ;
%Straight Lines [id:da7504189572934693] 
\draw [color={rgb, 255:red, 144; green, 19; blue, 254 }  ,draw opacity=1 ]   (318,23) -- (343,48) ;
%Straight Lines [id:da22228280567552094] 
\draw [color={rgb, 255:red, 144; green, 19; blue, 254 }  ,draw opacity=1 ]   (338,24) -- (363,49) ;
%Straight Lines [id:da41766582498852245] 
\draw [color={rgb, 255:red, 144; green, 19; blue, 254 }  ,draw opacity=1 ]   (358,23) -- (383,48) ;
%Straight Lines [id:da6891193548574022] 
\draw [color={rgb, 255:red, 144; green, 19; blue, 254 }  ,draw opacity=1 ]   (378,23) -- (403,48) ;
%Straight Lines [id:da15577158357178056] 
\draw [color={rgb, 255:red, 144; green, 19; blue, 254 }  ,draw opacity=1 ]   (398,23) -- (423,48) ;
%Straight Lines [id:da19106490994924996] 
\draw [color={rgb, 255:red, 144; green, 19; blue, 254 }  ,draw opacity=1 ]   (418,23) -- (443,48) ;
%Straight Lines [id:da9215014619009112] 
\draw [color={rgb, 255:red, 144; green, 19; blue, 254 }  ,draw opacity=1 ]   (438,23) -- (463,48) ;
%Straight Lines [id:da9188866107944152] 
\draw [color={rgb, 255:red, 144; green, 19; blue, 254 }  ,draw opacity=1 ]   (458,23) -- (483,48) ;
%Shape: Rectangle [id:dp6940878564896397] 
\draw   (58,48) -- (367,48) -- (367,72.17) -- (58,72.17) -- cycle ;
%Straight Lines [id:da23234678833722855] 
\draw [color={rgb, 255:red, 65; green, 117; blue, 5 }  ,draw opacity=1 ]   (367,88) -- (367,37) ;
%Straight Lines [id:da6317169268491063] 
\draw    (58,89) -- (58,38) ;
%Straight Lines [id:da882780468175289] 
\draw [color={rgb, 255:red, 65; green, 117; blue, 5 }  ,draw opacity=1 ]   (58,72.17) -- (81.9,48.27) ;
%Straight Lines [id:da6822201025383395] 
\draw [color={rgb, 255:red, 65; green, 117; blue, 5 }  ,draw opacity=1 ]   (78,72.17) -- (101.9,48.27) ;
%Straight Lines [id:da8212216159875206] 
\draw [color={rgb, 255:red, 65; green, 117; blue, 5 }  ,draw opacity=1 ]   (98,72.17) -- (121.9,48.27) ;
%Straight Lines [id:da32974615624742176] 
\draw [color={rgb, 255:red, 65; green, 117; blue, 5 }  ,draw opacity=1 ]   (118,72.17) -- (141.9,48.27) ;
%Straight Lines [id:da13968374887672397] 
\draw [color={rgb, 255:red, 65; green, 117; blue, 5 }  ,draw opacity=1 ]   (138,72.17) -- (161.9,48.27) ;
%Straight Lines [id:da9391683081037627] 
\draw [color={rgb, 255:red, 65; green, 117; blue, 5 }  ,draw opacity=1 ]   (158,72.17) -- (181.9,48.27) ;
%Straight Lines [id:da7196789908846116] 
\draw [color={rgb, 255:red, 65; green, 117; blue, 5 }  ,draw opacity=1 ]   (178,72.17) -- (201.9,48.27) ;
%Straight Lines [id:da5954258949768485] 
\draw [color={rgb, 255:red, 65; green, 117; blue, 5 }  ,draw opacity=1 ]   (198,72.17) -- (221.9,48.27) ;
%Straight Lines [id:da024517468203190784] 
\draw [color={rgb, 255:red, 65; green, 117; blue, 5 }  ,draw opacity=1 ]   (218,72.17) -- (241.9,48.27) ;
%Straight Lines [id:da7075768284731629] 
\draw [color={rgb, 255:red, 65; green, 117; blue, 5 }  ,draw opacity=1 ]   (238,72.17) -- (261.9,48.27) ;
%Straight Lines [id:da8143842151871976] 
\draw [color={rgb, 255:red, 65; green, 117; blue, 5 }  ,draw opacity=1 ]   (258,72.17) -- (281.9,48.27) ;
%Straight Lines [id:da9682563929158607] 
\draw [color={rgb, 255:red, 65; green, 117; blue, 5 }  ,draw opacity=1 ]   (278,72.17) -- (301.9,48.27) ;
%Straight Lines [id:da22273679305111327] 
\draw [color={rgb, 255:red, 65; green, 117; blue, 5 }  ,draw opacity=1 ]   (298,72.17) -- (321.9,48.27) ;
%Straight Lines [id:da3698715539647338] 
\draw [color={rgb, 255:red, 65; green, 117; blue, 5 }  ,draw opacity=1 ]   (318,72.17) -- (341.9,48.27) ;
%Straight Lines [id:da8088139366234447] 
\draw [color={rgb, 255:red, 65; green, 117; blue, 5 }  ,draw opacity=1 ]   (338,72.17) -- (361.9,48.27) ;
%Straight Lines [id:da4852287193863142] 
\draw [color={rgb, 255:red, 144; green, 19; blue, 254 }  ,draw opacity=1 ]   (496,85.52) -- (496,23) ;

% Text Node
\draw (431,89) node [anchor=north west][inner sep=0.75pt]  [color={rgb, 255:red, 208; green, 2; blue, 27 }  ,opacity=1 ] [align=left] {r};
% Text Node
\draw (303.78,91) node [anchor=north west][inner sep=0.75pt]   [align=left] {...};
% Text Node
\draw (328,89) node [anchor=north west][inner sep=0.75pt]   [align=left] {$\displaystyle b_{q_{1} ,q_{2} ,...,q_{l-1,} ,q_{l}}$};
% Text Node
\draw (453,89) node [anchor=north west][inner sep=0.75pt]   [align=left] {$\displaystyle b_{q_{1} ,q_{2} ,...,q_{l-1,} ,q_{l}-1}$};
% Text Node
\draw (216,89) node [anchor=north west][inner sep=0.75pt]   [align=left] {$\displaystyle b_{q_{1} ,q_{2} ,...,q_{l-1}}$};
% Text Node
\draw (178.78,91) node [anchor=north west][inner sep=0.75pt]   [align=left] {...};
% Text Node
\draw (52,96) node [anchor=north west][inner sep=0.75pt]   [align=left] {0};
% Text Node

\end{tikzpicture}
    \caption{$r$ between endpoints of intervals of $l$-th generation}
    \label{fig:r between endpoints}
\end{figure}

We have thus the following inequalities (see Figure \ref{fig:r between endpoints})
\[
m_{n}([ 0,r]) \leq m_{n}(\textcolor[rgb]{0.56,0.07,1}{[ 0,b_{q_{1} ,q_{2} ,...,q_{l-1,} ,q_{l}-1}]} )
\]
\[
|[ 0,r] |\geq |\textcolor[rgb]{0.25,0.46,0.02}{[ 0,b_{q_{1} ,q_{2} ,...,q_{l-1,} ,q_{l}}]} |
\]
The second observation gives the following
\\
\\
\resizebox{\textwidth}{!}{$
r^{h_n} \geq \left[\sum \limits_{j=q_1}^{\infty}(b_{j}-b_{j+1}) + (b_{q_1-1}-b_{q_1}) \sum \limits_{j=0}^{q_2-1}(b_{j}-b_{j+1}) +\dots + \prod \limits _{k=1}^{l-1}(b_{q_k-1}-b_{q_k})\sum  \limits_{j=q_l}^{\infty} (b_{j}-b_{j+1})\right]^{h_n}.
$}
Notice that this estimate for $r^{h_n}$ has the same form as the last component (C) in the estimate for $m_n([0,r])$, with the only difference being the last sum running from $j = q_{l}$ to infinity.
Hence
\begin{equation}\label{eq: measure div r}
\frac{m_n([0,r])}{r^{h_n}} \leq 
\end{equation}
\\
\\
\resizebox{\textwidth}{!}{$
A \cdot \frac{\left[\sum \limits_{j=q_1}^{\infty}(b_{j}-b_{j+1}) + (b_{q_1-1}-b_{q_1}) \sum \limits_{j=1}^{q_2-1}(b_{j}-b_{j+1}) +\dots + \prod \limits _{k=1}^{l-1}(b_{q_k-1}-b_{q_k})\sum  \limits_{j=q_l-1}^{\infty} (b_{j}-b_{j+1})\right]^{h_n}}{\left[\sum \limits_{j=q_1}^{\infty}(b_{j}-b_{j+1}) + (b_{q_1-1}-b_{q_1}) \sum \limits_{j=1}^{q_2-1}(b_{j}-b_{j+1}) +\dots + \prod \limits _{k=1}^{l-1}(b_{q_k-1}-b_{q_k})\sum  \limits_{j=q_l}^{\infty} (b_{j}-b_{j+1})\right]^{h_n}}=
$}
\[
= A \cdot B
\]
Here we denoted by B the remaining quotient. Focusing on A, we see that A is of the form:
\[
A =\sum\limits_{i = 1}^{t} w_i^{h_n},
\]
where $w_i$ is defined below in \Cref{defi: wi}.
We observe that in the numerator we have the sum of $t$ components of the form
\[
\prod \limits _{k=1}^{p}(b_{q_k-1}-b_{q_k}) \cdot (b_{j}-b_{j+1})
\]
each raised to the power $h_n$, and where $t = (n-q_1)+q_2+(n-q_3) + \ \dots \ + (n-q_l+1) \leq  n \cdot l$.
In the denominator, we have sum of the same components, but now the sum is raised to the power $h_n$. 
\begin{notation}\label{defi: wi}    
We denote $w_i$ as the $i$-th component from the sum in the nominator of the expression defining the value A, divided by the sum of all of those components.
\end{notation}
We immediately observe that sum of $w_i$ is equal to 1.
% For example 
% \\
% \\
% \resizebox{\textwidth}{!}{$
% w_1 =  \frac{b_{q_1}-b_{q_1+1}}{\sum \limits_{j=q_1}^{n-1}(b_{j}-b_{j+1}) + (b_{q_1-1}-b_{q_1}) \sum \limits_{j=1}^{q_2-1}(b_{j}-b_{j+1}) +\dots + \prod \limits _{k=1}^{l-1}(b_{q_k-1}-b_{q_k})\sum  \limits_{j=q_l}^{n-1} (b_{j}-b_{j+1})} ,
% $}
% \\
% \\
% \resizebox{\textwidth}{!}{$
% w_{n-q_1+2} =  \frac{ (b_{q_1-1}-b_{q_1})\cdot (b_0 - b_1)}{\sum \limits_{j=q_1}^{n-1}(b_{j}-b_{j+1}) + (b_{q_1-1}-b_{q_1}) \sum \limits_{j=1}^{q_2-1}(b_{j}-b_{j+1}) +\dots + \prod \limits _{k=1}^{l-1}(b_{q_k-1}-b_{q_k})\sum  \limits_{j=q_l}^{n-1} (b_{j}-b_{j+1})} ,
% $}
% \\
% \\
% \resizebox{\textwidth}{!}{$
% w_{t} =  \frac{ \prod \limits _{k=1}^{l-1}(b_{q_k-1}-b_{q_k})\cdot (b_{n-1}-b_{n})}{\sum \limits_{j=q_1}^{n-1}(b_{j}-b_{j+1}) + (b_{q_1-1}-b_{q_1}) \sum \limits_{j=1}^{q_2-1}(b_{j}-b_{j+1}) +\dots + \prod \limits _{k=1}^{l-1}(b_{q_k-1}-b_{q_k})\sum  \limits_{j=q_l}^{n-1} (b_{j}-b_{j+1})} .
% $}
Because $l$ is fixed (see the beginning of the proof) and $h_n \to 1$ as $n \to \infty$, then for sufficiently large n we get
\[
A = \sum\limits_{i = 1}^{t} w_i^{h_n} \leq t \cdot \frac{1}{t^{h_n}} = t^{1-h_n} \leq (n\cdot l) ^ {1 - h_n} = l^{1-h_n} \cdot n^{1-h_n}\leq  1 + \frac{\varepsilon}{2}.
\]
The first inequality comes from \Cref{lem: oszacowanie na sume wi}. The last inequality follows from the assumption that $n^{1-h_n} \to 1$ and $h_n \to 1$ as $n \to \infty$, and l is fixed. For sufficiently large $n$ then, both of $l^{1-h_n}$ and $n^{1-h_n}$ are less than or equal to $1+\frac{\varepsilon}{5}$ which proves that the inequality holds.
\\
Focusing now on the second part of the product
\\
\\
\resizebox{\textwidth}{!}{$
 B = \left [\frac{\sum \limits_{j=q_1}^{\infty}(b_{j}-b_{j+1}) + (b_{q_1-1}-b_{q_1}) \sum \limits_{j=1}^{q_2-1}(b_{j}-b_{j+1}) +\dots + \prod \limits _{k=1}^{l-1}(b_{q_k-1}-b_{q_k})\sum  \limits_{j=q_l-1}^{\infty} (b_{j}-b_{j+1})}{\sum \limits_{j=q_1}^{\infty}(b_{j}-b_{j+1}) + (b_{q_1-1}-b_{q_1}) \sum \limits_{j=1}^{q_2-1}(b_{j}-b_{j+1}) +\dots + \prod \limits _{k=1}^{l-1}(b_{q_k-1}-b_{q_k})\sum  \limits_{j=q_l}^{\infty} (b_{j}-b_{j+1})}\right ]^{h_n} =
$}
\[
=\left[ \frac{b}{a} \right]^{h_n},
\]
where $[a,b]$ is the unique interval in $\mathcal{F}_l^n$ containing point $r$, see beginning of the proof. Now, in the beginning of the proof of this proposition, for a given $\varepsilon>0$, we selected integer $l_0$ such that for all integer $l>l_0$ and all $n \in \N$ and for every interval $[a,b] \in \mathcal{F}_l^n$, the following holds $\frac{b}{a}\leq 1+\frac{\varepsilon}{4}$. Thus we get 
\[
B = \left[ \frac{b}{a} \right]^{h_n} \leq \left[ 1+\frac{\varepsilon}{4} \right]^{h_n} \leq 1 + \frac{\varepsilon}{4}.
\]
Putting those results together, we get that there exists $n_0 \in \mathbb{N}$ such that for all $n \geq n_0$ and all $r \in (0,1)$
\[
\frac{m_n([0,r])}{r^{h_n}} \leq A \cdot B \leq (1+\frac{\varepsilon}{2})\cdot(1+\frac{\varepsilon}{4}) \leq 1 + \varepsilon,
\]
which ends the proof.
\end{proof}
 \subsection{Estimate on the intervals $[b_{p+q},b_p]$}
We proved the first of two essential propositions in the proof of the lower bound \eqref{eq: linear lower bound} - the fact that upper limit of the densities of intervals of the form $(0,r)$ is smaller than or equal to 1. The next step is to prove that the upper limit of the densities (defined in Definition \ref{density}) of the intervals $[b_{p+q},b_p]$ is at most equal to one. In this Proposition we only use the Condition \eqref{warunek3}.
\begin{prop}\label{l do l+q}
If
\[
\lim \limits_{n\to \infty} (1-h_n)\ln n = 1,
\]
then
\[
\limsup \limits_{n \to \infty} \left ( \sup \left \{ \frac{m_n([b_{p+q},b_{p}])}{(b_{p}-b_{p+q})^{h_n}}: 0 \leq p<p+q \right \} \right ) \leq 1.
\]
\end{prop}
\begin{proof}
Fix some $p \geq 0$, $q \geq 1$. If $p+q > n$ then we can replace interval $[b_{p+q},b_{p}]$ with interval $[b_n, b_p]$ and the measure of both intervals is the same and the diameter of the second one is smaller, thus having larger density. Hence, we can focus on the case when $p+q \leq n$. Then we have that
\[
 \frac{m_n([b_{p+q},b_{p}])}{(b_{p}-b_{p+q})^{h_n}} = \frac{\sum \limits_{k = p}^{p+q-1} (b_k-b_{k+1})^{h_n}}{(b_{p}-b_{p+q})^{h_n}}=\frac{\sum \limits_{k = p}^{p+q-1} (b_k-b_{k+1})^{h_n}}{\left(\sum \limits_{k = p}^{p+q-1}\left( b_k-b_{k+1}\right)\right)^{h_n}} 
 \]
 \[
 = \sum \limits_{k = p}^{p+q-1} \left (\frac{b_k-b_{k+1}}{\sum \limits_{k = p}^{p+q} \left (b_k-b_{k+1}\right)}\right )^{h_n} = \sum \limits_{k = p}^{p+q-1} w_k^{h_n},
 \]
 where 
\[
w_k = \frac{\left(b_k-b_{k+1}\right)}{\sum \limits_{k = p}^{p+q-1} \left(b_k-b_{k+1} \right)} 
\]
and, of course, $\sum \limits_{k = p}^{p+q-1} w_k = 1$ . Now, from the \Cref{lem: oszacowanie na sume wi}, we get
\[
\sum \limits_{k = p}^{p+q-1} w_k^{h_n} \leq \sum \limits_{k = p}^{p+q-1} \left ( \frac{1}{q} \right )^{h_n} = (q) \left ( \frac{1}{q} \right )^{h_n} = \left(q \right )^{1-h_n} \leq n^{1-h_n} \leq 1 + \varepsilon 
\]
for sufficiently large n applying Condition \eqref{warunek2} yield the last inequality.
\end{proof}
 
\subsection{Putting the estimates together} 
The final part of the proof focuses on extending the family of sets with upper limit of the density at most $1$, up to a point from which we can conclude that the density of intervals in the family of all closed intervals contained in $[0,1]$ is at most $1$. Based on this fact, we can conclude that lower limit of the Hausdorff measure of the sets $J_n$ in its dimension is equal to 1 as $n \to \infty$.  
\\
In this part of the proof we use both conditions \eqref{warunek2} and  \eqref{warunek3}.
For each $k \in \mathbb{N}$ let  
\[
U_L(g_k(0), r )  = [g_k(0) - r, g_k(0)],
\]
where  $r \in (0, |g_k(0)-g_k(1)|) $. Observe that $g_k(1) = g_{k+1}(0)$. As a reminder, $g_k$ is the inverse map of $f_k$ and thus $g_k:[0,1] \to [b_k, b_{k-1}]$ with $g_k(0) = b_{k-1}$.
\begin{lemma}\label{r do gk}
\[
\limsup \limits_{n \to \infty} \left (\sup \left \{ \frac{m_n(U_L(g_k(0), r ))}{{\rm diam }(U_L(g_k(0), r ))^{h_n}}: k \in \mathbb{N} \text{ and } r \in (0, |g_k(0) - g_{k+1}(0)|)\right \} \right ) \leq 1
\]
\end{lemma}
\begin{proof}
If $k > n$ then $m_n(U_L(g_k(0), r )) = 0$. Hence we can assume, that $k\leq n$. Since 
\[
g_k^{-1}(U_L(g_k(0), r )) = [0, \hat{r} ]
\]
for $\hat{r} = g_k^{-1}(r)$, from \Cref{cor: wlasnosc niezmienniczosci mn} we have that 
\[
\frac{m_n(U_L(g_k(0), r ))}{{\rm diam }(U_L(g_k(0), r ))^{h_n}} = \frac{m_n([0, \hat{r} ])}{\hat{r}^{h_n}}
\]
and thus using \Cref{dziobanie} ends the proof.
\end{proof}
Now let us introduce the set 
\[
U_R(g_k(1), r ) =[g_k(1), g_k(1) + r],
\]
where $r \in (0, |g_k(0)-g_k(1)|) $. Then the following holds
\begin{lemma}\label{gk do r}
\[
\limsup \limits_{n \to \infty} \left (\sup \left \{ \frac{m_n(U_R(g_k(1), r ))}{{\rm diam }(U_R(g_k(1), r ))^{h_n}}: k\in\mathbb{N} \text{ and } r \in (0,[g_k(0), g_k(0) - g_{k+1}(0)])\right \} \right ) \leq 1
\] 
\end{lemma}
\begin{proof}
If $k > n$ then $m_n(U_R(g_k(0), r )) = 0$. Hence we can assume that $k\leq n$. Set $\varepsilon>0$.
Since 
\[
g_k^{-1}(U_R(g_k(0), r )) = [\hat{r}, 1 ]
\]
for $\hat{r} = g_k^{-1}(r)$, we have that 
\[
\frac{m_n(U_R(g_k(1), r ))}{{\rm diam }(U_R(g_k(1), r ))^{h_n}} = \frac{m_n([\hat{r}, 1 ])}{(1-\hat{r})^{h_n}}.
\]
There exists $p \in \mathbb{N}$ such that $\hat{r}\in (g_p(1), g_p(0)]$. If $p = 1$ then $[\hat{r}, 1 ] = U_L(g_1(0), 1-\hat{r})$, because $g_1(0) = 1$ and by \Cref{r do gk} we end the proof. 
\\
Let us assume from now on that $p \geq 2$. Then the interval $[\hat{r}, 1 ]$ splits into the following two parts
\[
[\hat{r}, 1 ] = [\hat{r}, g_q(0)]\cup [g_q(0), 1].
\]
 We know that $ [\hat{r}, g_q(0)] = U_L(g_q(0), g_q(0)-\hat{r})$ and thus by  \Cref{r do gk} we know that 
\[
\limsup \limits_{n \to \infty} \left (\sup \left \{ \frac{m_n([\hat{r}, g_q(0)])}{{\rm diam }([\hat{r}, g_q(0)])^{h_n}}: p \in \mathbb{N}, \hat{r} \in [g_p(1),g_p(0)]\right \} \right ) \leq  1,
\]
whereas for the interval $[g_q(0), 1]$, using \Cref{l do l+q} with $l = 0$ and thus $l+q = q$, we get that 
\[
\limsup \limits_{n \to \infty} \left (\sup \limits_{q \in \mathbb{N}} \left \{ \frac{m_n([g_q(0), 1])}{{\rm diam }([g_q(0), 1])^{h_n}}\right \} \right ) \leq  1.
\]
Now from this inequality we know that there exists $N_\varepsilon \in \mathbb{N}$ such that for all $n \geq N_\varepsilon$, and all $q \leq n$ and all $\hat{r}\in [0,|g_k(0) - g_{k}(1)|]$
\begin{equation}\label{jedna gwiazdka}
\frac{m_n([\hat{r}, g_q(0)])}{{\rm diam }([\hat{r}, g_q(0)])^{h_n}} \leq 1+\varepsilon 
\end{equation}
and
\begin{equation}\label{dwie gwiazdki}  
\frac{m_n([g_q(0), 1])}{{\rm diam }([g_q(0), 1])^{h_n}} \leq 1+\varepsilon
\end{equation}
for $n\geq N_\varepsilon$.
Set $\Delta = [\hat{r}, 1 ]$, $w_1 = \frac{|[\hat{r}, g_q(0)]|}{|\Delta|} $, $w_2 = \frac{|[g_q(0), 1]|}{|\Delta|}$. Then $w_1 + w_2 = 1$ and using \Cref{lem: oszacowanie na sume wi} we obtain 
\[
w_1^{h_n} + w_2^{h_n} \leq 2^{1-h_n} ,
\]
\[
|[\hat{r}, g_q(0)]|^{h_n} +|[g_q(0), 1]|^{h_n} \leq 2^{1-h_n} |\Delta|^{h_n}.
\]
Thus, for each $\varepsilon > 0$ there exists $n_0$, such that for every $n > n_0$ 
\[
\frac{m_n(\Delta)}{|\Delta|^{h_n}} \leq 2^{1-h_n} \frac{m_n([\hat{r}, g_q(0)]) + m_n([g_q(0), 1]) }{|[\hat{r}, g_q(0)]|^{h_n} +|[g_q(0), 1]|^{h_n} }\leq 
\]
\[
\leq 2^{1-h_n} \max \left \{ \frac{m_n([\hat{r}, g_q(0)])}{|[\hat{r}, g_q(0)]|^{h_n}}, \frac{m_n([g_q(0), 1])}{|[g_q(0), 1]|^{h_n}}\right \} \leq 2^{1-h_n} (1+ \varepsilon)
\]
and $h_n \to 1$, when $n \to \infty$, so the whole expression is arbitrarily close to 1, which ends the proof.
\end{proof}
Finally, we are ready to prove the main technical Theorem of this part of the paper.
\begin{thm}
\[
\limsup \limits_{n \to \infty} \left (\sup \left \{ \frac{m_n([s,t])}{([s,t])^{h_n}}: 0\leq s < t \leq 1, (s,t) \cap J_n \neq \emptyset \right \} \right ) \leq 1
\] 
\end{thm}
\begin{proof}
Fix $\varepsilon>0$. 
Assume that there exists no $k \in \{1,2, \dots, n \}$ such that $s < g_k(0)< t$. This implies that there exists $k$ such that $[s,t] \subseteq [g_k(1), g_k(0)]$. Then, using \Cref{cor: wlasnosc niezmienniczosci mn}, we obtain
\[
\frac{m_n([s,t])}{([s,t])^{h_n}} = \frac{m_n([g_k^{-1}(s),g_k^{-1}(t)])}{([g_k^{-1}(s),g_k^{-1}(t)])^{h_n}}.
\]
Then either there exists $l \in \{1,2, \dots, n \}$ such that  $g_k^{-1}(s) < g_l(0)< g_k^{-1}(t)$ or we can repeat this operation until such $l$ exists.
So from now on we will assume that there exists $k \in  \{1,2, \dots, n \}$ such that $s < g_k(0)< t$. Now either one of those two cases is true
\begin{enumerate}
    \item[(1)] at least one of the points $g_{k+1}(0)$ or $g_{k-1}(0)$ is contained in interval $[s,t]$,
    \item[(2)] $g_{k+1}(0) \not\in [s,t]$ and $g_{k-1}(0) \not\in [s,t]$.
\end{enumerate} 
Let us focus on the first case scenario. There exists minimal $l \in  \{1,2, \dots, n \}$ and maximal $q \in \{1,2, \dots, n \}$ such that $s\leq b_{l+q} < b_l \leq t$. We can divide interval $[s,t]$ into three pieces
\[
[s,t] = [s,b_{l+q})\cup [b_{l+q},b_l) \cup [b_l,t].
\]
Now recalling \Cref{r do gk}, \ref{gk do r} and \Cref{l do l+q} we know that there exists $N_\varepsilon \in \mathbb{N}$ such that for all $l, q$, all $s\in [b_{l+q+1},b_{l+q}]$ and all $t \in [b_q, b_{q-1}]$
\begin{equation}
\frac{m_n([s,b_{l+q}))}{{\rm diam }([s,b_{l+q}))^{h_n}} \leq 1+\varepsilon ,
\end{equation}
\begin{equation}\label{doklejony z lewej}
\frac{m_n([b_{l+q},b_l))}{{\rm diam }([b_{l+q},b_l))^{h_n}} \leq 1+\varepsilon 
\end{equation}
and 
\begin{equation}
\frac{m_n([b_l,t])}{{\rm diam }([b_l,t])^{h_n}} \leq 1+\varepsilon 
\end{equation}
for $n\geq N_\varepsilon$. Now using similar reasoning to the proof of \Cref{gk do r}, with the only difference being three components instead of two, we have that 
\[
\frac{m_n([s,t])}{([s,t])^{h_n}} \leq 3^{1-h_n}(1+\varepsilon),
\]
which ends the proof for the first case.
\\
For the case when  $g_{k+1}(0) \not\in [s,t]$ and $g_{k-1}(0) \not\in [s,t]$, we set $[b_{l+q},b_l) = \emptyset$ and then we split our interval into two parts $[s,b_k)$ and $[b_k,t]$. Now, invoking \Cref{r do gk} and \ref{gk do r}, we see that 
\[
\frac{m_n([s,b_{k}))}{{\rm diam }([s,b_{k}))^{h_n}} \leq 1+\varepsilon ,
\]
\[
\frac{m_n([b_k,t])}{{\rm diam }([b_k,t])^{h_n}} \leq 1+\varepsilon .
\]
Now, using the same reasoning as in \Cref{gk do r}, we get
\[
\frac{m_n([s,t])}{([s,t])^{h_n}} \leq 2^{1-h_n}(1+\varepsilon),
\]
ending the proof.
\end{proof}
This theorem together with \Cref{gestosc lepsza} and \Cref{gora} now gives.
\begin{thm}\label{razem}
Let $S_n$ be the iterated function system defined in Definition~\ref{IFS Sn} fulfilling the following two conditions 
\begin{enumerate}
\item[(C1)] 
\[
\lim\limits_{n \to \infty} \left(1-h_n\right) \ln{n} = 0 ,
\]
where $h_n$ is the Hausdorff dimension of the limit set of the IFS $S_n$ generated by initial $n$ contractions $g_k$, $k = 1, \dots, n$, and 
\item[(C2)]
\[
\sup \limits_{k\in \mathbb{N}} \left \{\frac{b_k-b_{k+1}}{b_{k+1}} \right \} < \infty.
\]
\end{enumerate}
Then
\[
\lim\limits_{n\to \infty} H_{h_n}(J_n) = 1
\]
where $J_n$ is the limit set of the IFS $S_n$ and $H_h$ denotes Hausdorff measure in Hausdorff dimension $h$.
\end{thm}
\section{Examples}\label{sec: examples} 
In this section we will show the families of iterated function systems for which the assumptions (1) and (2) of \Cref{razem} hold. The condition \eqref{warunek3} is easy to check. However, condition \eqref{warunek2} is the more demanding one. To check the \eqref{warunek2} we need only rough estimate on $1-h_n$. In many cases, we can get exact asymptotics of $h_n$, but we do not need it for this application.
\par The estimate for Hausdorff dimension can be deduced from Theorem 6.2.3 from \cite{Mauldin_Urbanski_2003}, which will also appear in Part II \Cref{section: nonlinear examples}. Below we formulate this theorem for iterated function systems consisting of linear functions, together with its proof based on \cite{Mauldin_Urbanski_2003}. The proof in the linear case is very transparent.
\begin{thm}\label{thm: 623 mauldin urbanski wymiar}
Let $S_n$ be a linear IFS with $|g_k'| = b_{k-1} - b_{k} : = a_k$. Suppose there exists $t_0 < 1$ such that $\sum \limits_{k=1}^\infty a_k^t < \infty$ for all $t \geq t_0$. Then
\[
1-h_n \leq \frac{\sum\limits_{k = n+1}^\infty |a_k|^t}{\sum\limits_{k = 1}^\infty (\ln \frac{1}{a_k})|a_k|}.
\]
for any $t_0\leq t \leq h_n$ where $h_n$ is the Hausdorff dimension of the limit set of $S_n$.
\end{thm}
\begin{proof}
Put $P(t) = \ln \left ( \sum\limits_{k=1}^\infty a_k^t\right)= \ln \left ( \sum\limits_{k=1}^\infty |g_k'|^t\right)$, $t \in \mathcal{D}$. The domain $\mathcal{D} = \{t: \sum\limits_{k=1}^{\infty} a_k^t < \infty \}$. This is an interval $[t_0,\infty)$ (or $(t_0, \infty)$), and we assume that $t_0 < 1$.
\par Observe that the function $P(t)$ is convex in its domain $\mathcal{D}$. Indeed, let $a \in [0,1]$. Then 
\[
P(at+(1-a)s)  = \ln \left( \sum_{k = 1}^\infty a_k^{at + (1-a)s} \right) \leq aP(t) + (1-a) P(s).
\]
This follows from H\"older inequality
\[
\sum x_i y_i \leq \left( \sum x_i^p \right)^{\frac{1}{p}} \left(\sum y_i^q \right)^{\frac{1}{q}}
\]
where $x_i = a_k^{at}$, $y_i = a_k^{(1-as)}$, $p = \frac{1}{a}$ and $q = \frac{1}{1-a}$.
\par The function $P(t)$ is differentiable in $(t_0,\infty)$. The derivative is given by the formula
\[
\frac{d}{dt} P(t) = \frac{\frac{d}{dt}\sum\limits_{k=1}^\infty a_k^t}{\sum\limits_{k=1}^\infty a_k^t} = \frac{\sum\limits_{k=1}^\infty a_k^t \ln (a_k) }{\sum\limits_{k=1}^\infty a_k^t}.
\]
Indeed, it follows from the fact that (due to the assumption $t_0 <1$) the series $\sum\limits_{k=1}^\infty a_k^t \ln (a_k)$ is locally uniformly convergent in $(t_0, \infty)$. Since the function $P(t)$ is convex, the derivative $P'(t)$ is increasing.
$P'(t) = \frac{d}{dt} P(t)$ is negative for $t <1$ and 
\[
P'(t) \leq P'(1) = \frac{\sum\limits_{k=1}^\infty a_k \ln (a_k) }{\sum\limits_{k=1}^\infty a_k} = - \int\ln(|f'|) d\lambda
\]
where $\lambda$ is the Lebesgue measure.
\par Now, we can also consider the function $P_n(t) := \ln \left( \sum\limits_{k=1}^{n} a_k^t  \right)$. Of course $P_n(h_n) = 0$ and $P(1) = 0 $. Thus
\[
e^{P(h_n)} - 1 = e^{P(h_n)} - e^{P(1)} = \int\limits_1^{h_n} \left( e^{P(t)} \right )' dt =   \int\limits_{h_n}^{1} -P'(t) e^{P(t)}  dt \overset{(*)}{\geq} -P'(1)  \int\limits_{h_n}^{1}  e^{P(t)}  dt \overset{(**)}{\geq} -P'(1)(1-h_n)
\]
The inequalities $(*)$ and $(**)$ are consequences of two facts:
\begin{enumerate}
    \item $-P'(t) \geq P'(1)$ because of convexivity of $P(t)$
    \item $e^{P(t)} \geq 1$ since $\sum\limits_{k=1}^\infty a_k^t \geq 1$ for all $t \in (t_0, 1)$.
\end{enumerate}
On the other hand
\[
e^{P(h_n)} - 1 = e^{P(h_n)} - e^{P_n(h_n)} = \sum\limits_{k=n+1}^\infty a_k^{h_n} \leq  \sum\limits_{k=n+1}^\infty a_k^{t} .
\]
for every $t \leq h_n$. Thus
\[
1-h_n \leq \frac{-1}{P'(1)} \cdot  \sum\limits_{k=n+1}^\infty a_k^{t}  = \frac{1}{\int\ln|f'| d\lambda}\cdot \sum\limits_{k=n+1}^\infty a_k^{t} 
\]
for any $t \leq h_n$. Hence
\[
1-h_n \leq \frac{\sum\limits_{k = n+1}^\infty |a_k|^t}{\sum\limits_{k =1}^\infty (\ln \frac{1}{a_k})|a_k|}
\]
for any $t \leq h_n$, which concludes the proof.
\end{proof}
This theorem can be applied for example for $b_k = \frac{1}{(k+1)^\alpha}$, $\alpha > 0$ and $ b_k = q^k$, $q \in (0,1)$ to verify condition \eqref{warunek2}.
\begin{lemma}
For $b_k = q^k$, $0<q<1$, as well as $b_k = \frac{1}{(k+1)^\alpha}$, $\alpha > 0$ the assumptions \eqref{warunek2} and \eqref{warunek3} hold an this \Cref{razem} applies.
\end{lemma}
\begin{proof}
In both cases, condition \eqref{warunek3} is straightforward to check. We will focus on showing that the assumption \eqref{warunek2} holds. 
\par Case 1. First, set $b_k = q^k$, $q \in (0,1)$. Remember that $a_k = b_{k-1} - b_k = q^{k-1}(1-q)$. Set $0<t \leq h_n$. Then, by \Cref{thm: 623 mauldin urbanski wymiar} we have
\[
1-h_n  \leq \frac{\sum\limits_{k = n+1}^\infty |a_k|^t}{\sum\limits_{k = 1}^\infty (\ln \frac{1}{a_k})|a_k|} = \frac{\sum\limits_{k = n+1}^\infty \left(q^{k-1}(1-q)\right)^t}{\frac{-q}{1-q}\ln(q)\ln(1-q)} = c\cdot \left(\frac{1-q}{q}\right)^t \cdot \sum\limits_{k = n+1}^\infty q^{tk} = 
\]
\[
= c\cdot \left(\frac{1-q}{q}\right)^t \cdot\frac{q^{t(n+1)}}{1-q^t}
\]
for some constant $c>0$. Hence
\[
\lim \limits_{n \to \infty} (1-h_n) \ln n \leq \lim \limits_{n \to \infty} q^{t(n+1)} \cdot\ln (n) \cdot  c \cdot\frac{1}{q^t} = 0,
\]
which ends the proof of Case 1.
\par Case 2. Set $b_k = \frac{1}{(k+1)^\alpha}$, $\alpha > 0$. Then, by the mean value theorem there exists $ \xi \in (k, k+1)$ such that
\[
a_k = b_{k-1} - b_k = \frac{1}{(k)^\alpha} - \frac{1}{(k+1)^\alpha} = -(k+1-k) \cdot \frac{-\alpha}{\xi^{\alpha+1}} \leq \frac{\alpha}{k^{\alpha+1}}
\] 
and hence $\frac{\alpha}{k^{\alpha+1}}\geq a_k \geq \frac{\alpha}{(k+1)^{\alpha+1}}$.
Set $t < 1$ such that $t (\alpha +1) >1$. For sufficiently large $n$, we have $t < h_n$. Now, using \Cref{thm: 623 mauldin urbanski wymiar} we get
\[
1-h_n  \leq \frac{\sum\limits_{k = n+1}^\infty |a_k|^t}{\sum\limits_{k =1}^\infty (\ln \frac{1}{a_k})|a_k|} \leq \frac{\sum\limits_{k = n+1}^\infty \left(\frac{\alpha}{k^{\alpha+1}}\right)^t}{\sum\limits_{k=1}^{\infty}\left[((\alpha+1)\ln(k) - \alpha)\frac{\alpha}{(k+1)^{\alpha+1}}\right]} = d\alpha^t \sum\limits_{k = n+1}^\infty \frac{1}{k^{t(\alpha+1)}} \overset{(*)}{\leq} d\alpha^t \int \limits_{n}^\infty \frac{1}{x^{t(\alpha+1)}} dx = 
\]
\[
= d\alpha^t \frac{1}{n^{t(\alpha+1)  -1}} \frac{1}{t(\alpha+1)-1}.
\]
 for some constant $d =  \left[\sum\limits_{k=1}^{\infty}\left[((\alpha+1)\ln(k) - \alpha)\frac{\alpha}{(k+1)^{\alpha+1}}\right]\right]^{-1}>0$. Inequality $(*)$ holds because 
 \[
\int_{k-1}^{k} x^{-t(\alpha+1)}dx \;\ge\; \text{(width)}\cdot\min_{[k-1,k]}x^{-t(\alpha+1)}= 1\cdot k^{-t(\alpha+1)}=\frac{1}{k^{t(\alpha+1)}}.
\]
Now because $t(\alpha+1) -1 > 0 $, we get
\[
\lim \limits_{n \to \infty} (1-h_n) \ln(n) \leq  \lim \limits_{n \to \infty} \frac{\ln(n)}{n^{t(\alpha-1)  -1}} \cdot   \frac{d\alpha^t}{t(\alpha+1)-1} = 0 
\]
which concludes the proof.
\end{proof}

\part{Nonlinear systems}
In this part, we will derive similar estimate for the Hausdorff measure as in previous one, but in more general case - we do not assume $f_k$ to be linear. Now, we assume that for each $k \in \mathbb{N}$ $f_k(x):[b_k, b_{k-1}] \to [0,1]$ is a decreasing continuous function, such that 
\[
f_k(b_{k}) = 1  \text{ and } f_k(b_{k-1}) = 0
\]
where as in Part 1, $(b_k)_{k=1}^{\infty}$ is a decreasing sequence with $b_0=1$ and $\lim\limits_{k\to \infty} b_k = 0$.
\section{Estimate in nonlinear case}\label{section: non linear case}
Let $g_k$ denote the inverse map $f_k^{-1}$. We assume that $g_k:[0,1]\to g_k([0,1])$ is a diffeomorphism of a $C^2$ class, which can be extended to some neighborhood of the interval $[0,1]$.
\begin{dfn}\label{def: distortion}
If $g:\Delta_1 \to \Delta_2$ is a diffeomorphism, we define
\[
\kappa(g)|_{\Delta_1} := \sup \left \{ \frac{|g'(y)|}{|g'(x)| }: x, y \in \Delta_1 \right \}
\]
and call this number the distortion of the map $g: \Delta_1 \to \Delta_2$
\end{dfn}
\textbf{Assumptions.} In this section, we assume the following five conditions. First, the same condition as in the linear case
\begin{equation}\tag{\ref{warunek dodatkowy 3}}
    \lim\limits_{n \to \infty} \left(1-h_n\right) \ln{n} = 0.
\end{equation}
Second, a stronger version of condition \eqref{warunek3}.
\begin{equation}\tag{\ref{warunek dodatkowy 2}}
\lim \limits_{n \to \infty} \sup \limits_{k \geq n} \left \{\frac{b_k}{b_{k+1}} \right \} = 1 .
\end{equation}
We will also need additional conditions on the non-linearity of the system. Those conditions are
\begin{equation}\tag{\ref{eq: warunek 4}}
\text{for every $n$ in $\N$, }\ \bigg |\frac{g''_n(x)}{g_n'(x)} \bigg | < c \text{ for some constant } c \in (0,\infty) \text{ independent of $n$} ,
\end{equation}
\begin{equation}\tag{\ref{eq: warunek 5}}
\lim \limits_{n \to \infty} \kappa(g_n) = 1 
\end{equation}
and
\begin{equation}\tag{\ref{eq: warunek 6}}
|g'_n| < \alpha < 1 
\end{equation}
for all $n \in \{ 1, 2, \dots \}$, and $\alpha$ independent of $n$.\\
Assuming conditions \eqref{warunek dodatkowy 3} - \eqref{eq: warunek 6}, we will prove in \Cref{thm: limit of the measure} that $\lim \limits_{n \to \infty}H_{h_n}(J_n)  =  1$ where $J_n$ is the limit set of the iterated function system $S_n = \{ g_k \}_{k = 1}^n$, and $h_n$ is the Hausdorff dimension of the set $J_n$. Since each $S_n$, $n \geq 2$ is a finite conformal iterated function system consisting of at least two elements and satisfying Open Set Condition, we have the following well-known result (see \cite{Mauldin_Urbanski_2003}, theorem 4.2.11)
\[
0<H_{h_n}(J_n)< \infty .
\]
Let $m_n(A):= H_n^{-1}(J_n) \cdot H_n(J_n \cap A)$. It follows from \cite{Mauldin_Urbanski_2003} that $m_n$ is the unique probability $h_n$ -conformal measure on $J_n$, where t-conformal measure $m$ is defined as follows 
\begin{equation}\label{eq: conformal measure}    
m(g_\omega(A)) = \int\limits_A |g_\omega'|^{t} dm
\end{equation}
for every Borel set $A \subset [0,1]$ and all $\omega \in \N^* := \bigcup_{n \geq 1} \N^n$. 
Now, we will state a lemma with an easy to prove property regarding the concept of distortion.
\begin{lemma}\label{lem: distortion lemmas}
Let $I_1, I_2$ be two intervals in $\mathbb{R}$. Let $\mathbb{D}(I_1, I_2)$ be the set of all diffeomorphisms from $I_1$ onto $I_2$. 
Assume $f \in \mathbb{D}(I_1, I_2)$ and $I$ is an interval contained in $I_1$, then
\[
\kappa^{-1}(f)|f'(x)| \cdot |I| \leq |f(I)| \leq \kappa(f) |f'(x)| |I|
\]
for every $x \in I_i$. In particular
\[
\kappa^{-1}(f)\sup\{|f'|\} \cdot |I| \leq |f(I)| \leq \kappa(f) \inf\{|f'|\} |I|.
\]
\end{lemma}
We will need one more, rather standard lemma.
\begin{lemma}\label{lem: distortion lemma additional}
Let $g_k \in C^2([0,1])$ for $k \in \{1, 2, 3 \dots \}$. Assume there exists $\alpha < 1$ such that $|g_k'(x)| < \alpha$ for each $k\in \{1, 2, 3 \dots \}$, $x \in [0,1]$. Moreover assume there exists $c > 0$ such that
\[
\left|\frac{g_k''(x)}{g_k'(x)}\right| < c
\]
for each $k\in \{1, 2, 3 \dots \}$, $x \in [0,1]$.
Then, there is a constant $D>0$ such that
\[
\left|\frac{g_\omega'(x)}{g_\omega'(y)}\right| \leq 1 + \frac{c}{1-\alpha}|x-y| + D(|x-y|^2).
\]
for $\omega \in \N^*$.
\end{lemma}
\begin{proof}
We include the proof for completeness. Let $\omega = [i_{1}, i_{2}, \dots, i_{k}]$. Then
\[
\left| \log \frac{|g_\omega'(x)|}{|g_\omega'(y)|}\right| = |\log |(g_\omega'(x))| - \log(|g_\omega'(y)|) | =
\]
\[
= \left| \sum\limits_{k = 1}^{|\omega|} \log |g_{i_k}'(g_{i_{k-1}} g_{i_{k-2}} \dots g_{1}(x))| - \sum\limits_{k = 1}^{|\omega|} \log |g_{i_k}'(g_{i_{k-1}} g_{i_{k-2}} \dots g_{1}(y))| \right |  \leq 
\]
\[
\leq \sum\limits_{k = 1}^{|\omega|} \left| \log |g_{i_k}'(g_{i_{k-1}} g_{i_{k-2}} \dots g_{1}(x))| - \log |g_{i_k}'(g_{i_{k-1}} g_{i_{k-2}} \dots g_{1}(y))| \right |
\]
Now, using intermediate value theorem and fact that $\log |(g_i'(x))|' = \frac{g_i''(x)}{g_i'(x)}$ we get
\[
\sum\limits_{k = 1}^{|\omega|} \left| \log |g_{i_k}'(g_{i_{k-1}} g_{i_{k-2}} \dots g_{1}(x))| - \log |g_{i_k}'(g_{i_{k-1}} g_{i_{k-2}} \dots g_{1}(y))| \right | \leq 
\]
\[
\leq \sum\limits_{k = 1}^{|\omega|} \left |\max\limits_{z \in [0,1]} \frac{|g_{i_k}''(z)|}{|g_{i_k}'(z)|}\right| \cdot \left| g_{i_k}\circ g_{i_{k-1}}\circ  g_{i_{k-2}} \dots g_{1}(x) - g_{i_k}\circ g_{i_{k-1}} \circ g_{i_{k-2}} \dots g_{1}(y) \right |.
\]
Since $|g_k'(x)| \leq \alpha$ for each $k \in \{1,2 \dots \}$, we get
\[
\left| \log \frac{|g_\omega'(x)|}{|g_\omega'(y)|}\right| \leq c \cdot \sum\limits_{k = 1}^{|\omega|} \left| g_{i_k}\circ g_{i_{k-1}}\circ  g_{i_{k-2}} \dots g_{1}(x) - g_{i_k}\circ g_{i_{k-1}} \circ g_{i_{k-2}} \dots g_{1}(y) \right | \leq
\]
\[
\leq c \cdot \sum\limits_{k = 1}^{|\omega|} \alpha^k |x-y| \leq \frac{c}{1-\alpha} |x-y|.
\]
And thus
\[
\frac{|g_\omega'(x)|}{|g_\omega'(y)|}\leq e^{\frac{c}{1-\alpha} |x-y|}\leq 1 + \frac{c}{1-\alpha}|x-y| + D(|x-y|^2),
\]
for some constant $D > 0$, independent of $\omega$, which concludes the proof.
\end{proof}
\subsection{Estimate from below}
% The general proof for the estimate from below of the lower limit of the densities is divided into two parts.
In this part, we want to show that
\[
\liminf\limits_{n\to\infty}H_{h_n}(J_n) \geq 1.
\]
\par We do it in several steps.
\par At first we estimate the lower limit of densities for intervals with diameter at least $\delta >0$ and next for the intervals of the form $[0,r]$ with $r \to 0$. We extend the estimate to the family of intervals to the intervals which have one of its endpoints as $g_k(0)$ for fixed $k \in \N$. 

\par Secondly, we expand on the set of intervals for which we can estimate the lower limit of densities, up to a point where we can deduce that lower limit of densities over any subinterval of $[0,1]$ is at least 1, using lemmas proven in the first part.

\par The general strategy of the proof is similar to the one proposed by \cite{UZ}. We significantly changed the proof to accommodate for more general case, however the general idea stays similar. Moreover, our approach makes the proof significantly easier.

\par We start by showing lemma saying that the lower limit of densities is at least 1 on the intervals that have diameter at least $\delta >0$. 
\begin{lemma} \label{lem:duzeprzedzialy}
For every $\delta > 0$ let $I_\delta$ be a family of all closed intervals $\Delta \subset [0,1]$ such that $|\Delta| \geq \delta$. Then for every $\varepsilon>0$ there exists $n_0 \in \N$ such that for every $n \geq n_0$ and $\Delta \in I_\delta$ we have
\begin{equation}\label{Rdelta}
\frac{m_n(\Delta)}{|\Delta|^{h_n}} \leq 1 + \varepsilon.
\end{equation}
\end{lemma} 
\begin{proof} Assume the contrary that there exists $\delta > 0$ such that the family $I_\delta$ does not fulfill \eqref{Rdelta}. This means that there exists $\eta \in [0,1)$, an increasing sequence $(n_j)^\infty_{j=1},\ n_j \in \mathbb{N}$ and a sequence $(\Delta_j)^\infty_{j=1}$ of closed intervals in $I_\delta$ such that 
\[
\lim\limits_{j \to \infty} \frac{m_{n_j}(\Delta_j)}{|\Delta_j|^{h_{n_j}}} = \frac{1}{\eta}.
\]
Now to continue this proof we will need a standard, auxiliary lemma letting us estimate the upper limit of the density of long intervals.
\begin{lemma}\label{lem: sequence mn weakly convergent to leb}
The sequence of measures $m_n$ converges weakly - * to the Lebesgue measure $m$ on the interval $[0,1]$.
\end{lemma}
\begin{proof}
First, we recall Theorem 3.16 from \cite{urbanski-mauldin}.
\begin{theorem*}[Theorem 3.16 in \cite{urbanski-mauldin}]
If $S = \{\phi_i: i \in I\}$ is a regular system, then $\lim\limits_{F \in Fin}m_F = m_I$ in the weak-* topology on C(X), where $Fin$ is a finite subset of $I$.
\end{theorem*}
In the theorem, the authors define a regular system as a conformal iterated function system that admits a t-conformal measure (see Definition \ref{eq: conformal measure}) and satisfies $m(\phi_i(X)\cap\phi_j(X)) = 0$ for every pair $i \neq j$, $i,j \in I$. In our setting, the measure $m_I$ is the Lebesgue measure $m$. Moreover, for $k\neq j$ the intersection $g_k([0,1])\cap g_j([0,1])$ contains at most one point, and hence has the Lebesgue measure zero. Consequently, Theorem 3.16 of \cite{urbanski-mauldin} applies and implies that the sequence $m_n$ converges, in the weak - * topology, to the unique conformal measure, which in this case is Lebesgue measure $m$. This ends the proof of \Cref{lem: sequence mn weakly convergent to leb}.

\end{proof}
We continue the proof of \Cref{lem:duzeprzedzialy}.
Modifying $(n_j)_{j=1}^\infty$, one can assume that the sequence of the intervals $\Delta_j = [a_j, b_j] \to [a,b] = \Delta$ as $j \to \infty$, where $\Delta$ is a interval contained in $[0,1]$.
Take an arbitrary closed interval $\hat{\Delta}$ such that $\Delta \subset \operatorname{int}\hat{\Delta}$ and $\frac{|\hat{\Delta}|}{|\Delta|}< \frac{1}{\beta}$, where $\eta < \beta<1$. Then
\begin{enumerate}
    \item for all $j$ large enough $\Delta_j \subset \hat{\Delta}$,
    \item $\lim \limits_{j \to \infty}m_{n_j}(\hat{\Delta})=m(\hat{\Delta})$ since $m(\partial\hat{\Delta})=0$,
    \item $m_{n_j}(\Delta_
    j)\leq m_{n_j}(\hat{\Delta})$.
\end{enumerate}
Thus
\[
1 = \frac{m(\hat{\Delta})}{|\hat{\Delta}|} = \frac{\lim\limits_{j \to \infty}m_{n_j}(\hat{\Delta})}{|\hat{\Delta}|} \geq \frac{\limsup\limits_{j \to \infty}m_{n_j}(\Delta_j)}{|\hat{\Delta}|} = \limsup \limits_{j \to \infty}\frac{|\Delta_j|^{h_{n_j}}}{|\hat{\Delta}|} \frac{m_{n_j}(\Delta_j)}{|\Delta_j|^{h_{n_j}}} \geq \frac{\beta}{\eta} > 1.   
\]
% Now, since $\frac{m_{n_j}(\Delta_j)}{|\Delta_j|^{h_{n_j}}} \to \frac{1}{\eta} > 1$ and $\frac{|\hat{\Delta}|^{h_{n_j}}}{|\hat{\Delta}|} \to \frac{|\Delta|}{|\hat{\Delta}|} > \beta$ with $j \to \infty$, 
This leads to a contradiction, proving \Cref{lem:duzeprzedzialy}.
% Taking a subsequence, we can assume that left and right endpoints of the intervals $\Delta_j$ converge to $a \in [0,1]$ and $b \in [0,1]$ accordingly and  $b-a \geq \delta$. Let $\Delta := [a,b]\in I_\delta$.
% Since the sequence $(m_{n_j})^\infty_1$ converges weakly to m, the Lebesgue measure, we get
% \[
% 1 = \frac{|\Delta|}{m(\Delta)} \leq \frac{\lim_{j\to \infty}|\Delta_j|^{h_{n_j}}}{\limsup_{j\to \infty}m_{n_j}(\Delta_j)} = \liminf\limits_{j \to \infty} \frac{|\Delta_j|^{h_{n_j}}}{m_{n_j}(\Delta_j)} = \eta < 1
% \]
% which gives contradiction, proving \Cref{lem:duzeprzedzialy}.
\end{proof}
Equipped with this lemma, we can prove an analogue to \Cref{dziobanie}.
\begin{lemma}\label{prop: uproszczony nieliniowy}
Let IFS fulfills the conditions
\eqref{warunek dodatkowy 3}, \eqref{warunek dodatkowy 2}, \eqref{eq: warunek 4} and \eqref{eq: warunek 5}. Then, for every $\varepsilon >0$ there exists $n_0 \in \N$ and $r_0 >0$ such that if $r \in (0,r_0)$ and $n> n_0$ then
\[
 \frac{m_n([0,r])}{r^{h_n}}\leq 1 + \varepsilon.
\]
\end{lemma}
\begin{proof}
Fix $\varepsilon >0$. If $r \leq b_{n}$, then $m_n([0,r]) = 0$. So, let $r > b_{n}$. There exists unique $k\leq n$ such that $b_{k}< r \leq b_{k-1}$. Moreover, from \eqref{eq: warunek 5} there exists $k_0 \in \N$ such that $\kappa(g_k) < 1 + \frac{\varepsilon}{4}$ and $\frac{b_k}{b_{k+1}}< 1 + \frac{\varepsilon}{4}$ for every $k \geq k_0$, thus let r be small enough that $b_{k}< r \leq b_{k-1}$ with $k \geq k_0$. Then
\begin{equation}\label{estymacja mn}
m_n([0,r]) \leq m_n([0,b_{k-1}]) \leq \sum\limits_{j = k}^{n} ||g_j'||_{\infty}^{h_n} m_n([0,1]) \leq 
\end{equation}
\[
\leq \sum\limits_{j = k}^{n} \kappa(g_j)^{h_n} \left( b_{j-1} - b_{j}\right)^{h_n} \leq  \sum\limits_{j = k}^{n} (1+ \frac{\varepsilon}{4})^{h_n} \left( b_{j-1} - b_{j}\right)^{h_n}.
\]
On the other hand we get
\begin{equation}\label{estymacja r}
r \geq \frac{b_{k}}{b_{k-1}} \cdot b_{k-1} \geq \frac{b_{k}}{b_{k-1}} \cdot \sum \limits_{j=k}^{n}(b_{j-1} - b_{j})
    % |[0,r]|^{h_n} = r^{h_n} \geq \left(\sum\limits_{j = k+1}^{\infty} b_j - b_{j+1}\right)^{h_n}  = b_{k+1}^{h_n}.
\end{equation}
Now, putting those two estimates together we obtain
\begin{equation}\label{eq: last inequality 0,r}
\frac{m_n([0,r])}{|[0,r]|^{h_n}} \leq \frac{(1+ \frac{\varepsilon}{4})^{h_n}\sum\limits_{j = k}^{n} \left( b_{j-1} - b_{j}\right)^{h_n}}{ b_{k}^{h_n}} 
\leq(1+ \frac{\varepsilon}{4})^{h_n} \cdot\sum _{j = k} ^{n } w_j^{h_n} \cdot \left (\frac{b_{k-1}}{b_{k}}\right )^{h_n},
\end{equation}
where $w_j = \frac{b_{j-1}-b_{j}}{\sum\limits_{j=k}^{n}(b_{j-1}-b_{j})}$.
Now, it follows from \Cref{lem: oszacowanie na sume wi} that the sum $\sum _{j = k} ^{n} w_j^{h_n}$ is at most $(\frac{1}{(n-k+1)})^{1-h_n}$, thus 
 \[
 (1+ \frac{\varepsilon}{4})^{h_n} \sum _{j = k} ^{n } w_j^{h_n} \leq (n-k+1) \cdot \frac{1}{(n-k+1)^{h_n}} \leq(1+ \frac{\varepsilon}{4})^{h_n} n^{1-h_{n}}.
 \]
As for the second part of the product in \eqref{eq: last inequality 0,r}: $\left (\frac{b_{k-1}}{b_{k}}\right )^{h_n}$, we know from Condition \eqref{warunek dodatkowy 2} and \eqref{warunek dodatkowy 3} that we can chose $n_0 \in \N$ such that for all $n \geq n_0$ we have $\left (\frac{b_{k-1}}{b_{k}}\right )^{h_n}< 1 + \frac{\varepsilon}{4}$ and $(1+ \frac{\varepsilon}{4})^{h_n} n^{1-h_{n}} < 1 + \frac{\varepsilon}{3}$.
Thus
\[
\frac{m_n([0,r])}{|[0,r]|^{h_n}} \leq (1+ \frac{\varepsilon}{4})^{h_n} n^{1-h_{n}}(1 +\frac{\varepsilon}{4}) \leq 1 + \varepsilon
\]
for all $n \geq n_0$ and all $r \leq b_{k_0}$.
\end{proof}
As a consequence of \Cref{lem:duzeprzedzialy} and \Cref{prop: uproszczony nieliniowy}, we get the following lemma. It differs from \Cref{prop: uproszczony nieliniowy}, as now we prove it for arbitrary $r \in (0,1]$.
\begin{lemma}\label{lem: estymacja na 0 r}
Let IFS fulfills the conditions
 \eqref{warunek dodatkowy 3}, \eqref{warunek dodatkowy 2}, \eqref{eq: warunek 4} and \eqref{eq: warunek 5}. Then, for every $\varepsilon >0$ there exists $n_0 \in \N$ and such that if $n> n_0$ then for every $r \in (0,1]$ we have
\[
 \frac{m_n([0,r])}{r^{h_n}}\leq 1 + \varepsilon.
\]
\end{lemma}
\begin{proof}
Fix $\varepsilon>0$. By \Cref{prop: uproszczony nieliniowy}, there exists $n_1 \in \N$ and $r_\varepsilon >0$ such that if $r \in (0,r_\varepsilon)$ and $n> n_1$ then
\[
 \frac{m_n([0,r])}{r^{h_n}}\leq 1 + \varepsilon.
\]
Now, from \Cref{lem:duzeprzedzialy} there exists $n_2 \in \N$ such that for every $n \geq n_2$ and $r \geq r_\varepsilon$
\[
 \frac{m_n([0,r])}{r^{h_n}}\leq 1 + \varepsilon.
\]
Now it suffices to notice that setting $n_0= \max\{n_1, n_2\}$ yields
\[
\frac{m_n([0,r])}{r^{h_n}} \leq \max \left \{ \sup\limits_{[0,r_\varepsilon]}\left \{ \frac{m_n([0,r])}{r^{h_n}}\right \}, \sup\limits_{[r_\varepsilon, 1]} \left \{ \frac{m_n([0,r])}{r^{h_n}}\right \} \right \}  \leq 1+ \varepsilon,
\]
which ends the proof.
\end{proof}
Fix $k \in \N$. Denote by $K(k,r)$ the interval 
\[
K(k,r) = [b_{k-1} - r, b_{k-1}], \text{ where } r \in (0,1).
\]
Denote by $K(k)$ the collection of intervals:
\[
K(k) = \{K(k,r): r \in (0,b_{k-1}) \}
\]
% For every $\omega \in \N^*$ let 
% \[
% K(g_\omega(0),r) = [g_\omega(0), g_\omega(0)+r] \text{ and } K(g_\omega(0),r) = [g_\omega(0)-r, g_\omega(0)]
% \]
% respectively if $|\omega|$ is even or odd. Let $K(\omega)$ be the collection of these intervals
% \[
% K(\omega) = \left \{ K(g_\omega(0),r): r \in (0,1] \right\}.
% \]
% Thus for a fixed $\omega$, $K(\omega)$ is a set that contains all intervals that have one of its endpoints as $g_\omega(0)$ and the intersection with the interval $g_\omega([0,1])$ is an interval.
\begin{lemma}\label{lem: fixed omega}
Fix $k \in \N$. Then for every $\varepsilon >0$ there exists $n_k^- \in \N$ such that for each $n \geq n_k^-$
\[
 \frac{m_n(\Delta)}{|\Delta|^{h_n}} \leq 1 + \varepsilon
\]
for every $\Delta \in K(k)$.
\end{lemma}
\begin{proof}
Note that this proof is for fixed $k \in \N$. For all $r \in (0, |g_k([0,1])|]$ there exists a unique $\hat{r} \in (0,1]$ such that
\[
K(k,r) = g_k([0,\hat{r}]).
\]
By \Cref{lem: distortion lemmas}, we get
\[
\kappa^{-1}(g_k|_{[0,\hat{r}]})|g'_k(0)|\hat{r} \leq r \leq \kappa(g_k|_{[0,\hat{r}]})|g'_k(0)|\hat{r}.
\]
Hence
\[
\frac{m_n(K(k,r))}{|K(k,r)|^{h_n}}\leq \frac{\kappa^{h_n}(g_k|_{[0,\hat{r}]})|g'_k(0)|^{h_n}m_n([0,\hat{r}])}{\kappa^{-h_n}(g_k|_{[0,\hat{r}]})|g'_k(0)|^{h_n}\hat{r}^{h_n}} = \kappa^{2h_n}(g_k|_{[0,\hat{r}]}) \frac{m_n([0,\hat{r}])}{\hat{r}^{h_n}}
\]
Since $\hat{r} \to 0$ as $r \to 0$, then as a consequence of \Cref{prop: uproszczony nieliniowy} and \ref{lem: distortion lemma additional}, we get that for every $\varepsilon > 0$ there exists an integer $N$ and a radius $\hat{R} \in (0, 1]$ such that
\[
\kappa^{2h_n}(g_k|_{[0,\hat{r}]}) \frac{m_n([0,\hat{r}])}{\hat{r}^{h_n}} \leq 1 + \varepsilon
\]
for all $n \geq N$ and all $0< \hat{r} \leq \hat{R}$. There exists a unique $R \in (0,b_{k-1}- b_k)$ such that $[b_{k-1}-R, b_{k-1}] = g_k([0,\hat{R}])$.  Using \Cref{lem:duzeprzedzialy}, we get
\[
\frac{m_n(K(k,r))}{r^{h_n}} \leq
\]
\[
\max \left \{\sup_{r \in (0,R]} \left \{ \frac{m_n(K(k,r))}{r^{h_n}}\right \},\sup_{r \in (R,b_{k-1}]} \left \{ \frac{m_n(K(k,r))}{r^{h_n}}\right \} \right \} \leq  1 + \varepsilon
\]
the lemma follows.
\end{proof}
\begin{lemma}\label{lem: endpoint in g omega}
Fix an arbitrary $k \in \N$, let $E(k)$ be the collection of all closed intervals $\Delta \subset [0,1]$ having $b_{k}$ as one of its endpoints. Then for every $\varepsilon>0$ there exists $n_k \in \N$ such that for all $n \geq n_k$ and every $\Delta \in E(k)$
\[
 \frac{m_n(\Delta)}{|\Delta|^{h_n}}\leq 1 + \varepsilon.
\]    
\end{lemma}
\begin{proof}
In view of \Cref{lem: fixed omega}, it is enough to consider intervals of the form $\Delta = [b_{k}, b_{k} +r]$, $k \in \N$. If $r \in (0,b_{k-1}-b_k)$, then $\Delta = g_k([1-R,1])$ for some $R:= R(r,k) \in (0,1)$.
\par Now, each interval $[1-R, 1]$ is of the form $[b_0-R, b_0]$, i.e. it is in the family $K(1)$ considered in \Cref{lem: fixed omega}. We know from \Cref{lem: fixed omega} that for all $n\geq n_1^-$, $R \in (0,1)$
\[
\frac{m_n([1-R,1])}{|[1-R,1]|^{h_n}} < 1 + \frac{\varepsilon}{4}.
\]
Using \Cref{lem: distortion lemma additional} for the map $g_1$ and \Cref{lem: distortion lemmas} we can find $r_k>0$ such that for every interval $\Delta$ of the form $\Delta = [b_k, b_k+r]$, $r \leq r_k$ and every $n \geq n_1^-$
\[
\frac{m_n(\Delta)}{|\Delta|^{h_n}} < \kappa^{2h_n}(g_1|_{[1-R,1]}) \frac{m_n([1-R,1])}{R^{h_n}} \leq (1 + \frac{\varepsilon}{4})^2 \leq 1 + \varepsilon.
\]
Using now \Cref{lem:duzeprzedzialy} for $\delta := r_k$, we find $n_0 = n_0(r_k)$ such that for $n \geq n_0$ and every interval $\Delta$ of length $|\Delta| > \delta$
\[
\frac{m_n(\Delta)}{|\Delta|^{h_n}} < 1 + \varepsilon. 
\]
Put $n_k^+ = \max \{ n_1^-, n_0\}$. Finally, putting $n_k = \max \{ n_k^-, n_k^+\}$ concludes the proof.
\end{proof}
\begin{lemma}\label{lem: suma N przyleglych przedzialow}
Let $\Delta_1,\Delta_2, \dots \Delta_N \subset [0,1]$ be closed intervals for $j = 1\dots N$ such that $\Delta_{j+1}$ is adjacent to $\Delta_j$ in a way that left endpoint of $\Delta_{j+1}$ an the right endpoint of $\Delta_{j}$ coincide. Put $\Delta = \Delta_1\cup \Delta_2 \cup \dots \cup \Delta_N $. Then
\[
\frac{m_n(\Delta)}{|\Delta|^{h_n}}\leq N^{1-h_n}\max\limits_{1 \leq j \leq N} \frac{m_{n}(\Delta_j)}{|\Delta_j|^{h_n}}.
\]
\end{lemma}
\begin{proof}
Notice that
\[
\frac{m_n(\Delta)}{|\Delta|^{h_n}} \leq \frac{\sum \limits_{i = 1}^{N} m_n(\Delta_i)}{\sum \limits_{i = 1}^{N}|\Delta_i|^{h_n}} \cdot \frac{\sum \limits_{i = 1}^{N}|\Delta_i|^{h_n}}{|\Delta_1 \cup \dots \cup \Delta_N|^{h_n}} \leq N^{1-h_n}\max\limits_{1 \leq j \leq N} \frac{m_{n}(\Delta_j)}{|\Delta_j|^{h_n}}
\]
where the last inequality comes from the fact that 
\[
\frac{\sum \limits_{i = 1}^{N} m_n(\Delta_i)}{\sum \limits_{i = 1}^{N}|\Delta_i|^{h_n}} \leq \max\limits_{1 \leq j \leq N} \frac{m_{n}(\Delta_j)}{|\Delta_j|^{h_n}}
\]
 and 
\[
\frac{\sum \limits_{i = 1}^{N}|\Delta_i|^{h_n}}{|\Delta_1 \cup \dots \cup \Delta_N|^{h_n}} \leq N^{1-h_n}.
\]
\end{proof}

\begin{lemma}\label{lem: zawierajace bk}
Fix $k \in \N$. Let $E^*(k)$ be the collection of all closed intervals in $[0,1]$ containing the point $b_k$. Then, for every $\varepsilon>0$ there exists $n_k^* \in \N$ such that for every $n \geq n_k^*$ and every $\Delta \in E^*(k)$ 
\[
\frac{m_n(\Delta)}{|\Delta|^{h_n}} \leq 1 + \varepsilon.
\]
\end{lemma}
\begin{proof}
Proof follows immediately from \Cref{lem: endpoint in g omega} and \Cref{lem: suma N przyleglych przedzialow} for $N = 2$.
\end{proof}
Let $L_k$ be the family of intervals $[b_{k-1} - r, b_{k-1}]$, $r \in [0, b_{k-1}-b_k]$, $k \in \N$. Note that $L_k \subset K(k)$. Recall that in \Cref{lem: fixed omega} we obtained estimates for the family $K(k)$ with fixed $k$. \Cref{lem: sk-} below, gives in turn similar estimate, but now these are uniform estimates for all families $L_k$.

\begin{lemma}\label{lem: sk-}[Estimates for the intervals in the family $L$]
For every $\varepsilon>0 $ there exists $\hat{n} \in \N$ such that for every $n>\hat{n}$
\[
\frac{m_n(\Delta)}{|\Delta|^{h_n}} \leq 1 + \varepsilon
\]
for each $ \Delta \in L :=  \bigcup\limits_{k = 1}^{\infty}L_k$.
\end{lemma}
\begin{proof}
Fix $\varepsilon>0$. We can use the same argument as in the beginning of the proof of \Cref{lem: fixed omega} and get
\[
\frac{m_n\left( \left[b_{k-1} -r, b_k \right]\right)}{r^{h_n}} \leq \kappa^{2}(g_k|_{[0,\hat{r}]}) \frac{m_n([0, \hat{r}])}{\hat{r}^{h_n}} \leq \kappa^{2}(g_k) \frac{m_n([0, \hat{r}])}{\hat{r}^{h_n}} 
\]
for all $r \in [0, b_{k-1} - b_k]$. Now, from Condition \eqref{eq: warunek 5} there exists $k_0 \in \N$ such that for all $k \geq k_0$ we have $\kappa^{2}(g_k) \leq 1 + \frac{\varepsilon}{4}$. From \Cref{lem: estymacja na 0 r} there exists $n_0\in \N$  such that for all $n> n_0$ and all $\hat{r} \in (0,1)$
\[
\frac{m_n([0, \hat{r}])}{\hat{r}^{h_n}} \leq 1 + \frac{\varepsilon}{4}.
\]
Now, for each $k \leq k_0$, by \Cref{lem: endpoint in g omega}, there exists $n_k^- \in \N$ such that for all $n \geq n_k^-$
\[
\frac{m_n\left( \left[b_{k-1} -r, b_k \right]\right)}{r^{h_n}} \leq 1 + \varepsilon.
\] 
Now, setting $\hat{n} = \max\{n_0, n_1^-, \dots n_{k_0}^- \} $ ends the proof.
\end{proof}
Analogously, define $P_k$ as
\[
\left [ b_{k}, b_{k}+r\right] , r \in \left[ 0, b_{k-1} - b_{k}\right] .
\]
Then the following holds.
\begin{lemma}\label{lem: sk+}[Estimates for the intervals in the family $P$]
For every $\varepsilon>0 $ there exists $n_0 \in \N$ such that for every $n>n_0$
\[
\frac{m_n(\Delta)}{|\Delta|^{h_n}} \leq 1 + \varepsilon
\]    
for each $\Delta \in P := \bigcup\limits_{k = 1}^{\infty} P_k$.
\end{lemma}
\begin{proof}
Observe that for each $r \in \left[ 0, b_{k-1}- b_{k}\right] $ there exists $ \hat{r} \in \left [0,1 \right] $ such that $\left [ b_{k}, b_{k}+r\right] = g_k([1-\hat{r}, 1])$. In the same manner as in previous proof, we get the following
\[
\frac{m_n(\left [ b_{k}, b_{k}+r\right])}{r^{h_n}} \leq \kappa^{2}(g_k|_{[1-\hat{r}, 1]}) \frac{m_n([1-\hat{r},1])}{\hat{r}^{h_n}} \leq \kappa^{2}(g_k) \frac{m_n([1-\hat{r},1])}{\hat{r}^{h_n}}.
\]
Now recalling Condition \eqref{eq: warunek 5} and using the same argument as in proof of  \Cref{lem: sk-}, ends the proof.
\end{proof}
For every $k \geq 0 $, let
\[
S_k = \left \{ \left [ b_{l+q}, b_l \right]: k \leq l, q \geq 1 \right \}.
\]
We shall prove the following
\begin{lemma}\label{lem: mk+}
\[
\limsup \limits_{\substack{n \to \infty }} \sup \left\{\frac{m_n(\Delta)}{|\Delta|^{h_n}}: \Delta \in S := \bigcup \limits_{k = 0}^{\infty} S_k \right\} \leq 1
\]
\end{lemma}
\begin{proof}

Since $m_n([0,b_{n}]) = 0$ we want to show that
\[
\limsup \limits_{\substack{n \to \infty}} \sup \left \{\frac{m_n(\left [ b_{l+q}, b_l \right])}{|\left [ b_{l+q}, b_l \right]|^{h_n}}:q \geq 1, 0 \leq l \leq l+q \leq n \right\} \leq 1.
\]
Fix $\varepsilon > 0$. Fix $k_0\in \N$ and $n_0 \in \N$ large enough so that by Condition \eqref{eq: warunek 5} $\kappa(g_j) \leq 1 + \frac{\varepsilon}{4}$ for all $j \geq k_0$. Fix some $l \geq k_0$, $q \geq 1$. Then we have
\begin{equation}\label{eq: mn w mk+}
m_n(\left [ b_{l+q}, b_l \right])  \leq  \sum\limits_{j = l+1}^{l+q} ||g_j'||_{\infty}^{h_n} m_n([0,1]) \leq 
\end{equation}
\[
\leq \sum\limits_{j = l+1}^{l+q} \kappa(g_j)^{h_n} \left( b_{j-1} - b_{j}\right)^{h_n} \leq  \sum\limits_{j = l+1}^{l+q} (1+ \frac{\varepsilon}{4})^{h_n} \left( b_{j-1} - b_{j}\right)^{h_n}.
\]
This gives us the following estimate
\[
\frac{m_n([b_{l+q},b_{l}])}{(b_{l}-b_{l+q})^{h_n}} \leq \frac{\sum\limits_{j = l+1}^{l+q} (1+ \frac{\varepsilon}{4})^{h_n} \left( b_{j-1} - b_{j}\right)^{h_n}}{(b_{l}-b_{l+q})^{h_n}}=(1+ \frac{\varepsilon}{4})^{h_n}\frac{\sum \limits_{j = l+1}^{l+q} (b_{j-1}-b_{j})^{h_n}}{\left(\sum \limits_{j = l+1}^{l+q}\left( b_{j-1}-b_{j}\right)\right)^{h_n}} =
\]
\[
= (1+ \frac{\varepsilon}{4})^{h_n} \sum \limits_{j = l+1}^{l+q} \left (\frac{b_{j-1}-b_{j}}{\sum \limits_{j = l+1}^{l+q} \left (b_{j-1}-b_{j}\right)}\right )^{h_n} = (1+ \frac{\varepsilon}{4})^{h_n} \sum \limits_{j = l+1}^{l+q} w_j^{h_n},
\]
 where 
\[
w_j = \frac{\left(b_{j-1}-b_{j}\right)}{\sum \limits_{j = l+1}^{l+q} \left(b_{j-1}-b_{j} \right)} 
\]
and $\sum \limits_{j = l+1}^{l+q} w_j = 1$ . Now, using \Cref{lem: oszacowanie na sume wi}, we get
\[
\sum \limits_{j = l+1}^{l+q} w_j^{h_n} \leq \sum \limits_{j = l+1}^{l+q} \left ( \frac{1}{q} \right )^{h_n} = q \left ( \frac{1}{q} \right )^{h_n} =
\]
\[
= \left(q \right )^{1-h_n} \leq n^{1-h_n} \leq 1 + \frac{\varepsilon}{2} 
\]
for all $n \geq n_0$ and $k \geq k_0$. Applying Condition \eqref{warunek dodatkowy 3} in the last inequality yields result for $k \geq k_0$. For $k < k_0$ notice that all of the intervals have length of at least $\min\limits_{j = 1, \dots, k_0}|b_{j-1}- b_{j}|$, thus invoking \Cref{lem:duzeprzedzialy} we get $n_1 \in \N$ such that for all $n \geq n_1$ 
\[
m_n(\left [ b_{l+q}, b_l \right])  \leq 1 + \varepsilon.
\]
Setting $\hat{n} = \max \{n_0, n_1 \}$ allows for the inequality 
\[
\frac{m_n(\Delta)}{|\Delta|^{h_n}} \leq 1 + \varepsilon
\]
for all $n\geq n_0$ and $\Delta \in S$, which ends the proof.
\end{proof}
Set $C$ as the family of all closed subintervals of $[0,1]$ intersecting the set $\{b_j: j \in \{0, 1 \dots\}\}$. Then, based on \Cref{lem: sk-}, \Cref{lem: sk+}, \Cref{lem: mk+} and \Cref{lem: suma N przyleglych przedzialow} for $N = 3$ we get the following lemma.
\begin{lemma}\label{lem: ekstremalnosc n}[Estimates for the intervals in the family $C$]
For every $\varepsilon >0$ there exists $n_0 \in \N$ such that for each $n \geq n_0$ we have
\[
\frac{m_n(\Delta)}{|\Delta|^{h_n}}\leq 1 + \varepsilon
\]
for every $\Delta \in C$.
\end{lemma}
\begin{proof}
Fix $\varepsilon>0$. By \Cref{lem: sk-}, we know that there exists $n_1 \in \N$ such that for every $n > n_1 $ 
\[
\frac{m_n(\Delta_1)}{|\Delta_1|^{h_n}}\leq 1 + \frac{\varepsilon}{4}
\]
for every $\Delta_1 \in L$. By \Cref{lem: sk+} there exists $n_2 \in \N$ such that for every $n > n_2 $ 
\[
\frac{m_n(\Delta_2)}{|\Delta_2|^{h_n}}\leq 1 + \frac{\varepsilon}{4}
\]
for every $\Delta_2 \in P$. From \Cref{lem: mk+}, there exists $n_3 \in \N$ such that for every $n > n_3 $ 
\[
\frac{m_n(\Delta_3)}{|\Delta_3|^{h_n}}\leq 1 + \frac{\varepsilon}{4}
\]
for every $\Delta_3 \in S$. Now, from \Cref{lem: suma N przyleglych przedzialow} and from the fact that for every $\Delta \in C$, $\Delta = \Delta_1 \cup \Delta_2 \cup\Delta_3$, where $\Delta_1 \in L$, $\Delta_2 \in S$ and $\Delta_3 \in P$, we have
\[
\frac{m_n(\Delta)}{|\Delta|^{h_n}} \leq 3^{1-h_n} \cdot (1 + \frac{\varepsilon}{4}).
\]
Setting $\hat{n} \in \N$ large enough such that $3^{1-h_n} < (1 + \frac{\varepsilon}{4})$ and taking $n_0 = \max\{\hat{n}, n_1, n_2, n_3 \}$ concludes the proof.
\end{proof}
This leads us to our final proposition.
\begin{prop}\label{prop: final from above}
For every $\varepsilon > 0$ there exists $n_0 \in \N$ such that for every $n > n_0$ and $\Delta \in \mathcal{W}$, we have
\[
\frac{m_n(\Delta)}{|\Delta|^{h_n}} \leq 1 + \varepsilon
\]
where $\mathcal{W}$ is the family of all closed intervals in $[0,1]$.
\end{prop}
Proof of this proposition requires simple, auxiliary lemma for IFS fulfilling the conditions \eqref{warunek dodatkowy 3} - \eqref{eq: warunek 6}.
\begin{lemma}\label{lem: lemat pomocniczy do W}
For every $\varepsilon>0$ there exists $M>0$ such that if $I \subset [0,1]$ is an interval, $J = g_k(I)$ and $|I| > M |J|$, then $k$ is so large that the distortion of the map $g_k$ satisfies $\kappa(g_k) < 1 +\varepsilon$. 
\end{lemma}
\begin{proof}[Proof of \Cref{prop: final from above}]
Fix $\varepsilon \in (0,1)$ and choose $\delta >0$ such that the distortion of every map $g_\omega$ restricted to an interval of length $\delta$ is less than $1 + \frac{\varepsilon}{4}$. This is possible due to \Cref{lem: distortion lemma additional}. Choose $M = M(\varepsilon)$ satisfying \Cref{lem: lemat pomocniczy do W}. Let $\Delta \in \mathcal{W}$ be an arbitrary closed interval. If $m_n(\Delta) = 0$ then, obviously $\frac{m_n(\Delta)}{|\Delta|^{h_n}} < 1 + \varepsilon$. 
For $\Delta \in \mathcal{W}$ such that $|\Delta| > \frac{\delta}{M}$, we apply \Cref{lem:duzeprzedzialy} and find $n_1 \in \N$ such that for all $\Delta \in \mathcal{W}$ with $|\Delta| > \frac{\delta}{M}$ and $n > n_1$
\[
\frac{m_n(\Delta)}{|\Delta|^{h_n}} \leq 1 + \varepsilon.
\]
Thus, we may assume that $|\Delta| \leq \frac{\delta}{M}$ and $m_n(\Delta) >0$.

\par Denote by $k = k(\Delta)$ the least integer such that $f^k(\Delta)$ contains some point $b_j$, $j \in \N$, and put $\Delta^{(1)} = f^{k}(\Delta)$. Note that $f^k: \Delta \to \Delta^{(1)}$ is a continuous bijection and its inverse is a restriction of some map of the form $g_\omega: g_\omega(\Delta^{(1)}) = \Delta$. Note also that the interval $\Delta^{(1)}$ is in the family $C$ from \Cref{lem: ekstremalnosc n}.
Now, we consider two cases.
\par Case 1. $|\Delta^{(1)}| < \delta$. We will use the following observation.
\begin{observation}\label{obs: delta to delta 1}
If $|\Delta^{(1)}| < \delta$ then 
\[
\frac{m_n(\Delta)}{|\Delta|^{h_n}}:\frac{m_n(\Delta^{(1)})}{|\Delta^{(1)}|^{h_n}} \leq (1 +\frac{\varepsilon}{4})^{h_n} \leq 1 +\frac{\varepsilon}{4}.
\]
This follows from \Cref{lem: distortion lemma additional}. Indeed,
\[
m_n(\Delta) = \int\limits_{\Delta^{(1)}} |g_\omega'|^{h_n} dm_n \leq \sup \limits_{x \in \Delta^{(1)}}|g_\omega'|^{h_n} m_n(\Delta^{(1)})
\]
and 
\[
|\Delta|^{h_n} = \left( \int\limits_{\Delta^{(1)}} |g_\omega'| dm \right)^{h_n}\geq |\Delta^{(1)}|^{h_n} \inf\limits_{x \in \Delta^{(1)}}|g_\omega'|^{h_n}
\]
where $m$ is the Lebesgue measure.
Hence, due to the choice of $\delta$,
\[
\frac{m_n(\Delta)}{|\Delta|^{h_n}}\leq \frac{m_n(\Delta^{(1)})}{|\Delta^{(1)}|^{h_n}} \left(\frac{\sup\limits_{x \in \Delta^{(1)}}|g_\omega'|}{\inf\limits_{x \in \Delta^{(1)}}|g_\omega'|} \right)^{h_n} \leq \frac{m_n(\Delta^{(1)})}{|\Delta^{(1)}|^{h_n}} (1 + \frac{\varepsilon}{4})^{h_n} \leq 1 + \frac{\varepsilon}{4}.
\]
This ends the proof of \Cref{obs: delta to delta 1}.
\end{observation} 
By \Cref{lem: ekstremalnosc n}, using the fact that $\Delta^{(1)} \in C$, there exists $n_2 \in \N$ such that for every interval in family $C$ (in particular $\Delta^{(1)}$) 
\[
\frac{m_n(\Delta^{(1)})}{|\Delta^{(1)}|^{h_n}} \leq 1 + \frac{\varepsilon}{4}
\]
for every $n > n_2$. Using \Cref{obs: delta to delta 1} we have
\begin{equation}
    \frac{m_n(\Delta)}{|\Delta|^{h_n}}:\frac{m_n(\Delta^{(1)})}{|\Delta^{(1)}|^{h_n}} \leq (1 +\frac{\varepsilon}{4})^{h_n} \leq 1 +\frac{\varepsilon}{4}.
\end{equation}
Hence, for $n > n_2$ we obtain
\[
\frac{m_n(\Delta)}{|\Delta|^{h_n}} \leq (1 +\frac{\varepsilon}{4})^2< 1 + \varepsilon
\]
for every interval $\Delta$ for which Case 1. holds.
\par Case 2. $|\Delta^{(1)}|\geq \delta$. Then denote by $l := l(\Delta) \leq k$ the least integer such that $|f^l(\Delta)|\geq \delta$. Denote $\Delta^{(2)} = f^l(\Delta)$ and $\Delta^{(3)} = f^{l-1}(\Delta)$. Then $|\Delta^{(3)}|< \delta$. Now, there are two subcases.
\par Subcase 2a. $|\Delta^{(3)}|< \frac{\delta}{M}$. This implies that the distortion of the map $f: \Delta^{(3)} \to \Delta^{(2)}$ is smaller than $1+\frac{\varepsilon}{4}$ by \Cref{lem: lemat pomocniczy do W}. Since the distortion of $f^{l-1}: \Delta \to \Delta^{(3)}$ is also smaller than $1+\frac{\varepsilon}{4}$ (by the choice of $\delta$, the same proof as in \Cref{obs: delta to delta 1}), we obtain that 
\begin{equation}
\frac{m_n(\Delta)}{|\Delta|^{h_n}}:\frac{m_n(\Delta^{(2)})}{|\Delta^{(2)}|^{h_n}} \leq (1 +\frac{\varepsilon}{4})^{h_n} \leq 1 +\frac{\varepsilon}{4}.
\end{equation}
Now, recall that there exists $n_3 \in \N$ such that for every interval $J$ of length at least $\delta$ and for every $n \geq n_3$ there holds
\[
\frac{m_n(J)}{|J|^{h_n}}\leq (1 + \frac{\varepsilon}{4})^{h_n}
\]
by \Cref{prop: uproszczony nieliniowy}. In particular it applies to $\Delta^{(2)}$, so for all $n \geq n_3$ we have 
\[
\frac{m_n(\Delta^{(2)})}{|\Delta^{(2)}|^{h_n}} \leq (1 + \frac{\varepsilon}{4})^{h_n}< 1 + \frac{\varepsilon}{4}.
\]
Thus,
\[
\frac{m_n(\Delta)}{|\Delta|^{h_n}} < (1 + \frac{\varepsilon}{4})^{3}  < 1 +\varepsilon
\]
for all $\Delta \in \mathcal{W}$ for which Subcase 2a. holds. %with $|\Delta^{(3)}|< \frac{\delta}{M}$ and $n \geq n_3$.
\par Subcase 2b. $|\Delta^{(3)}| \geq \frac{\delta}{M}$. Then, again by \Cref{prop: uproszczony nieliniowy}, there exists $n_4 \in \N$ such that for every interval with length $\geq \frac{\delta}{M}$ (in particular for $\Delta^{(3)}$ ), we have 
\[
\frac{m_n(\Delta^{(3)})}{|\Delta^{(3)}|^{h_n}} \leq (1 + \frac{\varepsilon}{4})^{h_n} < 1 + \frac{\varepsilon}{4}
\]
for all $n \geq n_4$. On the other hand, the length $|\Delta^{(3)}| < \delta$, so the distortion $f^{l-1}: \Delta \to \Delta^{(3)}$ is at most $1+\frac{\varepsilon}{4}$, and, as in previous cases, we conclude that 
\[
\frac{m_n(\Delta)}{|\Delta|^{h_n}} < 1 +\varepsilon
\]
for all $n \geq n_4$ and for all intervals $\Delta \in \mathcal{W}$ for which Subcase 2b. holds. Setting $n_0 = \max\{n_1, n_2, n_3, n_4\}$ concludes the proof of \Cref{prop: final from above}.

\end{proof}
\subsection{Estimate from above}
To complete the proof, we need one more estimate, namely
\begin{prop}\label{prop: nonlinear estimate from above}   
\[
\limsup\limits_{n\to\infty}H_{h_n}(J_n) \leq 1
\]
\end{prop}
This proposition is the nonlinear analogue of \Cref{gora}.
\begin{proof}
This part of the proof is considerably easier. First, note that every finitely generated iterated function system $S_n$ fulfills assumptions of well known fact that there exists $f-$invariant, ergodic probability measure $\mu_n$ equivalent to the conformal measure. Its proof can be found, for example, in the paper by Mauldin and Urbański \cite{urbanski-mauldin} - Theorem 3.8. This measure can be thought of both on the symbolic space, and on the set $J_n$. 
%§This is due to the fact, that on the symbolic space with finitely many symbols, the map $\{1, 2, \dots, n \}^\N \to J_n$ is continuous bijection (H\"older continuous with a natural metric on the symbolic space).
Now, the set $J_n$ is $f-$ invariant, meaning $f(J_n) = J_n$. This is obvious by the definition of the set $J_n$.
\par Thus for almost every $x \in J_n$ there exists infinitely many integers $k_j = k_j(x)$ such that $f^{k_j}(x) \in [b_{n}, b_{n-1}]$. This follows immediately from Birkhoff's ergodic theorem, the ergodicity of the measure $\mu_n$ and the fact that the set $[b_{n}, b_{n-1}]$ has positive measure. Hence there exists a Borel set $B_n \subset J_n$ with $\mu_n(B_n) = 1$ such that for every $x \in B_n$ there exists infinitely many $k_j(x)$ such that $f^{k_j}(x) \in [b_{n}, b_{n-1}]$. Fix $\varepsilon>0$. By \Cref{lem: distortion lemma additional} there exists $N_\varepsilon \geq 1$ such that
\[
\kappa(g_\omega|_{[b_{n,b_n-1}]}) < \frac{1}{1-\frac{\varepsilon}{3}}
\]
for every $n \geq N_\varepsilon$, every $x \in B_n$ and all $\omega \in \N^*$. Moreover, there exists $n_\varepsilon$ such that for all $n > n_\varepsilon$
\[
\frac{m_n([b_n, b_{n-1}])}{|b_n- b_{n-1}]^{h_n}} > 1 - \frac{\varepsilon}{3}.
\]
This is a consequence of the fact that $\kappa(g_n) \to 1$ as $n \to \infty$.
Take $n \geq n_1 = \max\{n_\varepsilon, N_\varepsilon\}$. Let $g_\omega^{(j)} = g_{i_1}\circ \dots \circ g_{k_j}$ be such a composition of maps $g_i$ such that for $x \in B_n$ $x \in g_\omega([b_n, b_{n-1}])$. Let $I_j(x)$ be an interval $I_j(x) = g_\omega([b_n, b_{n-1}])$.
Then, by apply \Cref{lem: distortion lemma additional}, we obtain
\[
\frac{m_n(I_j(x))}{|I_j(x)|^{h_n}} \geq \frac{\inf\limits_{[b_n, b_{n-1}]}|g_\omega'|^{h_n}}{\sup\limits_{[b_n, b_{n-1}]}|g_\omega'|^{h_n}} \cdot \frac{m_n([b_n,b_{n-1}])}{|b_n-b_{n-1}|^{h_n}} > (1-\varepsilon/3)^2 > 1- \varepsilon.
\]
Thus, we have demonstrated the existence of arbitrarily small interval containing $x$, such that the ratio of its measure to diameter raised to $h_n$ is arbitrarily close to 1 from below.
Hence, we can deduce that
\[
\limsup\limits_{r \to 0} \left \{ \frac{m_n(B_n)}{\operatorname{diam}^{h_n}(B_n)}: \operatorname{diam}(B_n) \leq r \right \} \geq 1 - \varepsilon.
\]
Because $m_n(B_n) = \mu_n(B_n) = 1$ due to \Cref{density}, this implies that $H_{h_n}(J_n) \leq \frac{1}{1 - \varepsilon}$ for all $n \geq n_1$.
\end{proof}
Now, we can finally combine \Cref{prop: nonlinear estimate from above} together with \Cref{prop: final from above} to get the final theorem.
\begin{thm}\label{thm: limit of the measure}
Let $J_n$ be the limit set of the IFS $S_n$ fulfilling the conditions \eqref{warunek dodatkowy 2} -  \eqref{eq: warunek 6}. Then
\[
\lim \limits_{n \to \infty} H_{h_n}(J_n) = 1
\]
where $h_n$ is the Hausdorff dimension of $J_n$ and $H_{h_n}$ is the Hausdorff measure of $J_n$ in its dimension.
\end{thm}
\section{Examples}\label{section: nonlinear examples}
In this section, we will provide a wide range of iterated function systems, which fulfill conditions \eqref{warunek dodatkowy 3} - \eqref{eq: warunek 6}. Most of those conditions are straightforward to check. The one which can be hard to check is the condition \eqref{warunek dodatkowy 3}, as we need an estimate on how fast the Hausdorff dimension approaches 1. This problem however is solved by Theorem 3.3 from \cite{urbanski-heinemann} or equivalently from Theorem 6.2.3 from \cite{Mauldin_Urbanski_2003}. We will cite the version from the latter source, adapted to our case.
\begin{thm}\label{thm: dimension asymptotics}
Let $S$ be a conformal iterated function system. Assume there exists $0<h_0<1$ such that for all $h_0<h\leq 1 $ the following holds
\[
\sum\limits_{k=1}^\infty ||g_k'||^h <  \infty.
\]
Then, for all $h \geq h_n$, there exists a constant $C>0$ independent of $n$, such that the following inequality is satisfied
\[
1-h_n \leq C \cdot \sum\limits_{k = n+1}^\infty ||g_k'||^h.
\]
\end{thm}
This theorem lets us fairly easily check if our IFS satisfies condition \eqref{warunek dodatkowy 3}, because
\[
\lim\limits_{n\to \infty} (1-h_n)\ln n \leq \ln(n) \cdot  \sum\limits_{k = n+1}^\infty ||g_k'||^h.
\]
Applying this result lets us check, that the subset of $[0,1]$ of all numbers whose infinite continued fraction expansions have all entries in the finite set $\{1, 2, \dots , n\}$, indeed fulfills the assumptions \eqref{warunek dodatkowy 3} - \eqref{eq: warunek 5} of \Cref{thm: limit of the measure}. Condition \eqref{eq: warunek 6} is not immediately fulfilled by this iterated functions system, as the derivative of Gauss map $|g_n'| = \frac{1}{(n+x)^2}$ is equal to 1 for $g_1(0)$. However, this is not an issue, as in the second iteration of this IFS does not have this issue and indeed fulfills \eqref{eq: warunek 6}.

\bibliographystyle{abbrv}
\bibliography{main}

\end{document}